\documentclass[11pt]{article}

\usepackage{hyperref}
\hypersetup{
    unicode=false,          
    colorlinks=true,        
    linkcolor=blue,          
    citecolor=darkgreen,        
    filecolor=magenta,      
    urlcolor=cyan           
}

\usepackage{xspace}
\usepackage{graphicx}
\usepackage{waingarten}
\usepackage{paralist}
\usepackage{thmtools}
\usepackage{thm-restate}
\usepackage{color}
\usepackage{setspace}

\def\FullBox{\hbox{\vrule width 8pt height 8pt depth 0pt}}
\newcommand{\QED}{\;\;\;\FullBox}
\renewenvironment{proof}{\noindent{{\textbf{Proof:}~}}} {\hfill\QED}
\providecommand{\email}[1]{\href{mailto:#1}{\nolinkurl{#1}\xspace}}

\def\FullBox{\hbox{\vrule width 8pt height 8pt depth 0pt}}

\newcommand{\cA}{{\cal A}}
\newcommand{\cB}{\mathcal{B}}

\newcommand{\cD}{\mathcal{D}}

\newcommand{\cM}{{\cal M}}
\newcommand{\cN}{{\cal N}}

\newcommand{\RR}{\mathbb R}

\newcommand{\var}{\mathrm{var}}

\newcommand{\norm}[1]{\left\lVert #1 \right\rVert}

\newcommand{\Exp}{\EX}

\newcommand{\EX}{\hbox{\bf E}}

\newcommand{\grad}{\nabla}

\newcommand{\Eqn}[1]{\hyperref[eq:#1]{(\ref*{eq:#1})}} 
\newcommand{\Thm}[1]{\hyperref[thm:#1]{Theorem\,\ref*{thm:#1}}} 
\newcommand{\Lem}[1]{\hyperref[lem:#1]{Lemma\,\ref*{lem:#1}}} 
\newcommand{\Clm}[1]{\hyperref[clm:#1]{Claim~\ref*{clm:#1}}} 

\title{Monotonicity Testing of High-Dimensional Distributions with Subcube Conditioning\footnote{This work was partially completed while authors were visiting the Simons Institute for the Theory of Computing at UC Berkeley. DC is supported by National Science Foundation (NSF) under grants CCF-2041920, CCF-2402571.
XC is supported   by NSF grants CCF-2106429, CCF-2107187. CS is supported by NSF grants CCF-1740850, CCF-1839317, CCF-2402572, and DMS-2023495.
EW is supported by NSF grant  CCF-2337993.}}

\author{
Deeparnab Chakrabarty\\
Dartmouth College\\
{\tt deeparnab@dartmouth.edu}
\and
Xi Chen\\
Columbia University\\
{\tt xichen@cs.columbia.edu}
\and
Simeon Ristic\\
University of Pennsylvania\\
{\tt sristic@seas.upenn.edu}
\and
C. Seshadhri\\
University of California, Santa Cruz\\
{\tt sesh@ucsc.edu}
\and
Erik Waingarten \\
University of Pennsylvania\\
{\tt ewaingar@seas.upenn.edu}
}

   \date{}
\begin{document}
\maketitle

\begin{abstract}
\ignore{Distribution testing has led to a plethora of results, but for most interesting properties, the query\xnote{Should this be sample complexity?} complexity
is polynomial in the domain size. For high-dimensional domains, such complexities are infeasible.
The subcube conditional model (Canonne-Ron-Servedio, SICOMP 2015 and Bhattacharyya-Chakraborty, TOCT 2018)
allows for sampling subcubes of the domain, often leading to queries polynomial in the \emph{dimension}.

We study the classic property of monotonicity of distributions over $\{0,1\}^n$, under the subcube conditional model.
Consider an unknown input distribution $p$ over $\{0,1\}^n$. A distribution is monotone if 
$p(x) \leq p(y)$ for any $x \preceq y$, where $\preceq$ denotes the standard coordinate-wise partial order over $\{0,1\}^n$.
The aim is to distinguish monotone $p$ from $p$ being $\eps$-far (in TV distance) from monotone.
Up to logarithmic factors, we resolve the query complexity to be $\Theta(n/\eps^2)$.

For the upper bound, we take inspiration from the literature of monotonicity testing of Boolean functions.
We run an ``edge tester" for testing distribution monotonicity. To bound the query complexity, we
rely on directed isoperimetric theorems that relate the "directed boundary" of functions
to their distance to monotonicity. We prove a real-valued analogue of the directed Talagrand theorem (Khot-Minzer-Safra, SICOMP 2018),
which is a key tool in our proof.
For the lower bound, we reduce to dealing with product distributions under the standard sampling model. 
We develop lower bounds for determining the bias of coordinates, which are then applied to distributional monotonicity testing.

Applying our methods, we also prove an $\tilde{\Omega}(\sqrt{n}/\eps^2)$ lower bound for uniformity testing
with subcube conditional queries, under the promise that the distribution is monotone. This matches the 
upper bound up to logarithmic factors, which holds for arbitrary distributions. Thus, it proves that monotonicity
does not help for uniformity testing with subcube conditional queries. }

We study monotonicity testing of high-dimensional distributions on $\{-1,1\}^n$ in the model of subcube conditioning, suggested and studied by Canonne, Ron, and Servedio~\cite{CRS15} and Bhattacharyya and Chakraborty~\cite{BC18}. Previous work shows that the \emph{sample complexity} 
of monotonicity testing must be exponential in $n$ 
(Rubinfeld, Vasilian~\cite{RV20}, and 
Aliakbarpour, Gouleakis, Peebles, Rubinfeld, Yodpinyanee~\cite{AGPRY19}). We show that the subcube \emph{query complexity} is $\tilde{\Theta}(n/\eps^2)$,
by proving nearly matching upper and lower bounds. Our work is the first to use directed isoperimetric inequalities (developed for function monotonicity testing) for analyzing a distribution testing algorithm. Along the way, we generalize an inequality of Khot, Minzer, and Safra~\cite{KMS18} to real-valued functions on $\{-1,1\}^n$.

We also study uniformity testing of distributions that are promised to be monotone, a problem introduced by Rubinfeld, Servedio~\cite{RS09}
, using subcube conditioning. We show that the query complexity is $\tilde{\Theta}(\sqrt{n}/\eps^2)$. Our work proves the lower bound, which matches (up to poly-logarithmic factors) the uniformity testing upper bound for general distributions (Canonne, Chen, Kamath, Levi, Waingarten~\cite{CCKLW21}).
Hence, we show that monotonicity does not help, beyond logarithmic factors, in testing uniformity of distributions with subcube conditional queries.

\end{abstract}
\thispagestyle{empty}
\newpage
\begin{spacing}{0.75}
\tableofcontents
\end{spacing}
\thispagestyle{empty}

\newpage

\pagenumbering{arabic}
\setcounter{page}{1}


\section{Introduction}

Distribution testing is the study of algorithms for testing properties of probability distributions \cite{GR00, BFRSW00}. A distribution testing problem is specified by a class of distributions $\calP$ 
supported on a domain $\Sigma$.  The aim is to get low-complexity algorithms that distinguish an unknown probability distribution $p$ having the property (i.e., $p \in \calP$) 
from one that is ``far" from having the property. One hallmark result is an algorithm and a matching lower bound, showing that $\Theta(\sqrt{|\Sigma|} / \eps^2)$ independent samples are necessary and sufficient for testing whether $p$ is the uniform distribution~(see, the recent survey~\cite{C22} and references therein for an overview of the area). For most problems of interest, such a polynomial dependence on the support size $\Sigma$ is intrinsic. This makes classical distribution testing algorithms intractable for high-dimensional distributions, such as those supported on the hypercube $\Sigma = \{-1,1\}^n$, where the complexity becomes exponential in the dimension.

To circumvent this issue, Bhattacharrya-Chakraborty~\cite{BC18} and Canonne-Ron-Servedio~\cite{CRS15} introduced the \emph{subcube conditional model}
for distributions supported on $\{-1,1\}^n$. An algorithm can query the distribution $p$ with a subcube $\rho \in \{-1,1, *\}^n$, and receive an independent sample $\bx \sim p$ conditioned on $\bx_i = \rho_i$ for all $i$ where $\rho_i \neq *$. We highlight some appealing aspects of this model.

\begin{flushleft}\begin{itemize}
\item For high-dimensional distributions, the subcube conditional model may provide an appropriate analogue to a ``membership query'' in learning theory, where distribution testing algorithms can overcome exponential dependencies in the dimension (e.g., see~\cite{CJLW21b, BLMT23}).
\item There is a growing body of work seeking to make high-dimensional distribution testing practical (particularly in the context of testing software that produces samples), where one can often implement more powerful query oracles, and in particular, the subcube conditional sampling oracle (e.g., see~\cite{MPC20, PM22,BCPSS24}).
\item Subcube conditioning lends itself to an elegant mathematical analysis, often leading to query complexities polynomial in the dimension (as opposed
    to domain size). The key to an efficient testing result often involves (approximately) determining a global property of the distribution while only estimating marginals after conditioning on (random) subcubes of the domain (e.g., see~\cite{CCKLW21, BCSV23, CM24, AFL24}).
\end{itemize}\end{flushleft}

In this paper, we focus on the classic property of monotonicity of distributions. We use $\preceq$ to denote the coordinate-wise partial order on $\{-1,1\}^n$ and use $p$ to denote the probability mass function.
A distribution $p$ supported on $\{-1,1\}^n$ is monotone if, for any pair $x, y \in \{-1,1\}^n$ with $x \preceq y$, 
$p(x) \leq p(y)$. 
Monotonicity arises naturally in many scenarios and is a desirable property from an algorithmic perspective~(e.g., \cite{BLMT23}). 
For any distribution $p$, we define the \emph{distance to monotonicity} as $\min_{q } \dtv(p,q)$, 
where the minimum is over all monotone distributions $q$ and $\dtv$ is the total variation distance between distributions. We say that $p$ is \emph{$\eps$-far from monotone}
if the distance to monotonicity is at least $\eps$.
A \emph{tester for monotonicity} gets access to a distribution $p$ and a proximity parameter $\eps \in (0,1)$. 
If $p$ is monotone, the tester accepts with probability $> 2/3$. If $p$ is $\eps$-far from monotone,
the tester rejects with probability $> 2/3$.
Our main question is: what is the query complexity of monotonicity testing of distributions over the hypercube $\{-1,1\}^n$, in the subcube conditional model?


%
%
%

\subsection{Our Contributions} \label{sec:contributions}

\paragraph{Testing Monotonicity.} Our first results are a testing algorithm and a nearly-matching lower bound which shows the query complexity of testing monotonicity in the subcube conditional model is \emph{linear} in the dimension.

\begin{theorem}[Monotonicity Testing Upper Bound]\label{thm:intro-ub}
There is an algorithm for testing monotonicity of distributions over $\{-1,1\}^n$ that uses $\tilde{O}(n/\eps^2)$ subcube conditioning queries. The algorithm works in the weaker \emph{coordinate oracle} model~\cite{BCSV23}, where queries are only made on one-dimensional subcubes.
\end{theorem}


\begin{theorem}[Monotonicity Testing Lower Bound]\label{thm:intro-lb}
Any monotonicity tester of distributions using subcube conditioning must use $\tilde{\Omega}(n/\eps^2)$ queries.
\end{theorem}

We will overview the proofs of \Thm{intro-ub} and \Thm{intro-lb} shortly. Roughly speaking, the upper bound will follow from exploiting a connection between monotonicity testing and directed isoperimetric inequalities, and defining an ``edge tester'' for distribution testing. For the lower bound, we construct a pair of distributions over product distributions which will ``hide'' the negative biases. 

    \noindent\textbf{Directed Isoperimetry.} The key tool enabling \Thm{intro-ub} is a directed isoperimetric inequality for real-valued functions. As we expand on in Section~\ref{sec:related}, isoperimetric inequalities relate the surface area of a geometric object to its volume. In the hypercube, the volume of a subset $A \subset \{-1,1\}^n$ is its size and the surface area is a measure of edges ``crossing'' the set. \emph{Directed} isoperimetry is the phenomenon where a non-monotone function exhibits evidence of its non-monotonicity via a directed edge boundary (points where the function value decreases). Directed isoperimetric inequalities relate the ``non-monotone edge boundary'' (the measure of surface area) to the distance to monotonicity (the measure of volume). These inequalities have played a major role in testing monotonicity of functions. In this work, we use a real-valued version of a (Boolean-valued) directed isoperimetric inequality of Khot-Minzer-Safra~\cite{KMS18} (also~\cite{PRW22}, who remove a logarithmic factor), which may be of independent interest.\footnote{As we discuss in Section~\ref{sec:related}, real-valued versions of directed isoperimetric inequalities have been studied~\cite{BKR23}, although we will need one of a different form.} Namely, let $\dist_1(f)$ denote the $\ell_1$-distance to monotonicity of any 
    $f:\{-1,1\}^n\rightarrow \mathbb{R}$, i.e., the minimum over monotone functions $g \colon \{-1,1\}^n \to \R$ of $\Ex_{\bx}[|f(\bx) - g(\bx)|]$.
\begin{theorem}\label{thm:l1-talagrand}
	For any $f:\{-1,1\}^n \to \R$, we have
 \[
	\Ex_{\bx \sim \{-1,1\}^n}\left[  \left( \sum_{i:\bx_i=-1} \left(\big(f(\bx) - f(\bx^{(i)})\big)^+\right)^2 \right)^{1/2}  \right] \geq \Omega\left(\frac{1}{\sqrt{\log n}}\right)\cdot \dist_1(f).
	\]
\end{theorem}
(We use $(z)^+$ as shorthand for $\max(z,0)$, and for $x \in \{-1,1\}^n$, $x^{(i)} \in \{-1,1\}^n$ refers to the point that agrees with $x$ on all but the $i$-th coordinate.)
The proof of this theorem is obtained by a reduction to the Boolean case, using a thresholding technique of Berman-Raskhodnikova-Yaroslavstev~\cite{BeRaYa14}. 
We note that this theorem answers a open question in~\cite{F23} 
(the question mark in Table 1, which asks for an $(L^1, \ell^2)$-Poinc\'{a}re theorem)).
It is an interesting open problem to determine whether the dependence on $n$ in \Thm{l1-talagrand} can be removed (no such dependence exists for the Boolean valued case).

\noindent\textbf{Uniformity of Monotone Distributions.} We then turn our attention to a problem introduced by Rubinfeld-Servedio, where they seek distribution testing algorithms under the promise that the underlying distributions are monotone~\cite{RS09}. 
They study the classic problem of uniformity testing, and show that $\tilde{O}(n/\eps^2)$ independent samples suffice (a significant improvement over the exponential sample lower bounds). 
With subcube conditioning, one would hope for further significant improvements. In particular, uniformity testing (without any assumption on the distribution) can be done in $\tilde{O}(\sqrt{n}/\eps^2)$ queries~\cite{CCKLW21}; does the complexity decrease further when the input distribution is promised to be monotone?
We prove a lower bound, showing that uniformity testing of monotone distributions requires $\tilde{\Omega}(\sqrt{n}/\eps^2)$ subcube conditioning queries, i.e., it is equally hard for general and monotone distributions. 

\begin{theorem}[Testing Uniformity of Monotone Distributions]\label{thm:uni-of-mon}
Testing uniformity of monotone distributions requires $\tilde{\Omega}(\sqrt{n}/\eps^2)$ subcube conditioning queries.
\end{theorem}

\ignore{\paragraph{Directed Isoperimetric Theorems.} One of the key tools for the upper bound of \Thm{intro-ub} is
a \emph{directed isoperimetric theorem} for real-valued functions. The problem of property testing monotonicity
of functions over $\{-1,1\}^n$ has seen a rich history of results. 
Much of the progress is due to a connection with directed isoperimetric inequalities~\cite{GGLRS00, ChSe13-j, KMS15, BCS23, BKKM23, F23, BKR23, F24}. 
Directed isoperimetry is the phenomenon that a function that is far from monotone must exhibit evidence via a directed boundary (points where the function value decreases)
The directed isoperimetric inequalities ensure that various functionals of the directed boundary are large whenever the function is far from monotone.
Our proof of \Thm{intro-ub} shows an interesting use of directed isoperimetry for distribution testing.

We need to prove a new directed isoperimetric theorem, which is of independent interest. 
We state some notation to formally state this theorem.

Let $f:\{-1,1\}^n \to [0,1]$ be a function defined on the $n$-dimensional hypercube. The $L_1$-distance 
of $f$ from monotonicity is defined as
\begin{equation*}
	\dist_1(f) \eqdef \min_{g~:~\text{monotone}} ~~\Ex_{\bx \sim \{-1,1\}^n}\left[ |f(\bx) - g(\bx)| \right]
\end{equation*}
where the expectation is over the uniform distribution over $\{-1,1\}^n$.
For a point $x\in \{-1,1\}^n$, define the directed derivative $\grad^-f(x)$ to be the $n$-dimensional vector defined as 
\begin{equation}\label{eq:defgrad}
	\left(\grad^-f(x)\right)_i \eqdef \begin{cases}
		0 & \textrm{if $x_i = 1$} \\
		\left(f(x) - f(x+2e_i)\right)^+ & \text{otherwise}
	\end{cases}
\end{equation}
where $(z)^+$ is a shorthand for $\max(z,0)$. For Boolean-valued $f:\{-1,1\}^n \to \{0,1\}$, the distance $\dist_1(f)$ corresponds to the ``normal'' Hamming distance notion, $\dist_0(f)$.
Based on isoperimetric theorems of Talagrand~\cite{Tal93}, the quantity $\Exp_{\bx} \norm{\grad^-f (\bx)}_2$ can be thought
of as a ``directed surface area" for the function $f$. A deep isoperimetric theorem of Khot, Minzer, and Safra~\cite{KMS15} (see, also~\cite{PRW22}, who showed how to remove the final logarithmic factor)
lower bounds this surface area by the distance to monotonicity.

\begin{theorem}[\cite{KMS15, PRW22}]\label{thm:booliso}
	There exists a universal constant $C > 0$ such that for every $f \colon \{-1,1\}^n \to \{0,1\}$,
$\Exp_{\bx} \norm{\grad^-f (\bx)}_2 \geq C\cdot \dist_0(f)$.
\end{theorem}

We prove a real-valued generalization of this statement, with a $\sqrt{\log n}$ loss in the bound.
\begin{theorem}\label{thm:l1-talagrand}
	For any $f:\{-1,1\}^n \to [0,1]$, 
 \[
	\Ex_{\bx \sim \{-1,1\}^n}\left[  \norm{\grad^-f (\bx)}_2 \right] \geq \Omega\left(\frac{1}{\sqrt{\log n}}\right)\cdot \dist_1(f).
	\]
\end{theorem}

The proof of this theorem is obtained by a reduction to the Boolean case of \Thm{booliso}, using a thresholding
technique of Berman-Raskhodnikova-Yaroslavtsev~\cite{BeRaYa14}. 
We note that this theorem answers an open question in~\cite{F23} 
(the question mark in Table 1, which asks for an $(L_1, \ell_2)$-Poinc\'{a}re theorem)).
An interesting open problem is to remove
the dependence on $n$ from \Thm{l1-talagrand}.}

\subsection{Main Ideas} \label{sec:ideas}
We give a short summary of the main ideas behind our results. 

\noindent\textbf{Monotonicity Testing Upper Bound.} The algorithm behind \Thm{intro-ub} is an ``edge tester" for distribution testing. We take a sample $\bx \sim p$ and choose a direction $\bi \sim [n]$ uniformly at random, and we consider the distribution $p$ conditioned on $\smash{\{\bx, \bx^{(\bi)}\}}$, or equivalently, on the 1-dimensional subcube $\brho$ which is $\bx$ except for $\brho_{\bi} = *$. Notice that the conditional distribution is supported on one bit $\{-1,1\}$. For our algorithm, it suffices to condition on 1-dimensional subcubes, which have been studied under the name \emph{coordinate oracles}~\cite{BCSV23}. In a monotone distribution, $1$ is more likely to appear than $-1$, so if the algorithm can detect (by few independent samples from the subcube) that the probability of $-1$ is larger than that of $1$, it can safely output ``reject'' (meaning non-monotone).
The challenge is understanding, when $p$ is $\eps$-far from monotone, how likely it is that an ``edge'' $\smash{\{ \bx, \bx^{(i)}\}}$ is biased toward $-1$, and is this bias detectable from few samples. 

The connection to directed isoperimetry then becomes clear:
considering the case when $p$ is \mbox{$\eps$-far} from monotone, we apply \Thm{l1-talagrand} on the probability mass function $p$, where $\dist_1(p)$ can be~shown to be $\Omega(\eps)$ (Corollary~\ref{cor:l1-tal}). Then, the left-hand side of \Thm{l1-talagrand} can be used to establish (via an averaging argument and an ``importance sampling'' trick) how large the bias toward $-1$ will be on a random draw of $\bx \sim p$ and $\bi \sim [n]$ (in the proof of Lemma~\ref{lem:far-case-reject}). We emphasize that, even though we seek lower bounds on the biases of individual edges, which a (simpler) directed Poincar\'{e} inequality would seemingly handle (namely, Theorem~1.3 in~\cite{F23}), such an argument only gives a weaker $O(n^2/\eps^2)$ complexity. The reason is the following: directed Poincar\'{e} lower bounds the sum of $(p(x) - p(x^{(i)}))^+$ across all edges $\{x, x^{(i)}\}$, and would imply an edge-wise bias of $\eps / n$ (since there are $n 2^n$ edges); however, the query complexity of detecting bias is inverse {\em quadratic} in the bias. On the other hand, \Thm{l1-talagrand} rules out such situations: if all edges had negative bias $\eps / n$, $p$ would be $\tilde{O}(\eps / \sqrt{n})$-close to monotone. 

\ignore{The negative
biased edgs are like a measure of ``surface area". The problem was basic isoperimetric theorems, like the Poincar\'{e}
inequality, is that they only bound the total sum of negative biases over all edges. The query complexity
is inverse \emph{quadratic} in the biases, so for an optimal complexity bound, we need to understand how
the negative bias is spread over the edges. We would prefer if this biased is concentrated on a smaller set 
of edges. Remarkably, the directed Talagrand inequality of \Thm{l1-talagrand} precisely asserts that. If the total
negative bias is small (close to minimum possible for $\eps$-far distributions), then this bias is concentrated
on a small set of edges. This ensure that, on sampling such an edge, the bias can be determined quickly.
Combining with a logarithmic search over bias values, one can prove that the edge tester works in $\tilde{O}(n/\eps^2)$ samples.}

\noindent\textbf{Monotonicity Testing Lower Bound.} Our starting point is to focus on proving sample-complexity lower bounds for testing the analogous problems (namely, monotonicity and uniformity of monotone distribution) in the restricted setting of \emph{product distributions}. This is a significant restriction from general distributions\footnote{Product distributions over $\{-1,1\}^n$ are fully specified by their mean vector, so in order to describe a product distribution, it suffices to specify the $n$ values of $\mathbb{E}_{\bx \sim p}[\bx_i]$. General distributions over $\{-1,1\}^n$, on the other hand, lie in the convex hull of $\{ e_i : i \in [2^n]\}$ and require $2^{n}-1$ numbers to specify.}, and these immediately imply general query lower bounds for subcube conditioning (Lemma~\ref{lem:tree-to-iid}). In particular, conditioning on subsets of coordinates (i.e., subcubes) do not change the distribution at all if coordinates were independent to begin with. The surprising aspect is that these lower bounds on product distributions will turn out to be nearly optimal (this was also the case in~\cite{CCKLW21, CJLW21b}).

Let us first focus on the monotonicity testing lower bound (\Thm{intro-lb}), where we use Yao's minimax
principle. Since a product distribution is fully specified by its mean vector, we construct a pair of distributions $\calA$ and $\calB$ supported on $[-1,1]$; and let the mean vector (and hence product distribution) have, for each $i \in [n]$, the $i$-th coordinate given by the $i$-th draw from $\calA$ or $\calB$ (Lemma~\ref{lem:one-dim-dist}). $\calA$ will always output non-negative numbers (so the corresponding product distribution will always be monotone), while $\calB$ will be $-\eps / \sqrt{n}$ with constant probability (so the corresponding product distribution will be $\Omega(\eps)$-far from monotone). The key trick is to design $\calA$ and $\calB$ in such a way so as to have their first $\log n / \log \log n$ moments being identical; in Section~\ref{sec:indist} we show how it implies that $\tilde{\Omega}(n/\eps^2)$ samples are needed to distinguish product distributions they produce.
The matching moments technique for lower bounds has been used before~\cite{RRSS09, V11}. Most recently, it was used to prove lower bounds for subcube conditioning for testing and learning $k$-juntas~\cite{CJLW21b}.


\noindent\textbf{Uniformity of Monotone Distributions.} The lower bound on testing uniformity of monotone distributions also proceeds by Yao's minimax principle. We note that, the fact that testing uniformity of product distributions required $\Omega(\sqrt{n}/\eps^2)$ samples was known~\cite{CDKS17}; however, the examples to obtain the lower bound were non-monotone (each coordinate behaves independently with bias set to $\pm \eps / \sqrt{n}$). In our construction, each of the $n$ coordinates is biased with probability $1/\sqrt{n}$, but its bias becomes $\eps / n^{1/4}$ (Section~\ref{sec:dno-def}). Note that all biases are positive, so the resulting distribution is monotone. 
%
In order to prove indistinguishability, we reduce to the corresponding problem under Gaussian distributions (instead of bits). The calculation
boils down to the behavior of sum of exponentials of Gaussians, which can be determined by the (closed form) expression for the moment generating function (Section~\ref{sec:proof-mono-lb}).

\subsection{Related Work}\label{sec:related}

As mentioned above, directed isoperimetry theorems have been  crucial to the development of testers for monotonicity of Boolean functions~\cite{GGLRS00,ChSe13-j,KMS18} (as opposed to monotonicity of distributions). In particular, as we illustrate in~\Thm{booliso}, the strongest known directed isoperimetry theorem for Boolean functions due to~\cite{KMS18} (and slightly improved by~\cite{PRW22}) relates the expected $\norm{\grad^-f(x)}_2$ to the $\ell_0$-distance of $f$ from monotonicity, that is, the fraction of domain points at which $f$ must be modified to make it a monotone function. Our result,~\Thm{l1-talagrand}, is an ``$\ell_1$-version'' of the above statement for {\em real-valued} functions over the Boolean hypercube. 

The most relevant works to \Thm{l1-talagrand} are the directed isoperimetry theorems initiated by Pinto Jr.~\cite{F23, F24} who considers smooth functions $f:[0,1]^n \to \RR$. In~\cite{F23}, Pinto Jr. looks at the $\ell_1$-geometry and proves under a certain $\ell_1$-smoothness condition, the expected $\ell_1$-norm of the gradient is at least the $\ell_1$-distance of $f$ from monotonicity. In the subsequent paper~\cite{F24}, Pinto Jr. assumes $\ell_2$-smoothness and proves that the expected $\ell_2^2$-norm of the gradient is at least the square of the $\ell_2$-distance. Neither of these results imply or are implied by the Boolean setting of~\cite{KMS18} or our result,~\Thm{l1-talagrand}. As mentioned earlier, our result,~\Thm{l1-talagrand}, answers a question left open in~\cite{F23}.
Using the notation of that paper, we prove an $(L^1, \ell^2)$-Poinc\'{a}re theorem
for real-valued functions on the hypercube.

 In~\cite{BKR23}, Black, Kalemaj and Raskhodnikova consider Boolean functions $f:\{-1,1\}^n \to \mathbb{R}$ and they prove that the isoperimetry result of~\cite{KMS18} generalizes for such functions in the following sense. Instead of looking at the $\ell_2$-norm of the (directed) gradient $\grad^- f(x)$,~\cite{BKR23} considers the square-root of the ``negative influence'' at each $x$, where the ``negative influence'' counts the number of pairs which form a monotonicity violation with $x$. The {\em magnitude} of the violation is ignored. 
In this setting,~\cite{BKR23} proves that if a real-valued function is $\eps$-far from being monotone in the $\ell_0$-sense (which is usual in property testing), then the expected square-root of the negative influence is $\Omega(\eps)$ thereby generalizing~\cite{KMS18}. The authors use this result to give an $O(r\sqrt{d}/\eps^2)$-query non-adaptive tester for real-valued monotone functions, where $r$ is the cardinality of the image of $f$. Our directed isoperimetry result seems unrelated to their result, apart from the fact that both of our results are proved by reducing it to the Boolean case. 
Finally, in~\cite{BCS23}, Black, Chakrabarty and Seshadhri generalize the directed isoperimetry theorem of~\cite{KMS18} to Boolean functions defined over the {\em hypergrid}. 

\paragraph{Monotonicity Testing of Distributions.} Monotonicity of distributions has been studied extensively in the literature, in both low-dimensional and high-dimensional regimes~\cite{BKR04, RS09, ACS10, BFRV11,AGPRY19,RV20}. Batu, Kumar, and Rubinfeld initiated the study in~\cite{BKR04} and considered both regimes above. They described a 
tester for one-dimensional distributions (total orders) using $\tilde{O}_\eps(\sqrt{n})$-samples, and via a reduction to uniformity testing proved a tightness of this result. They also proved a $\Omega(m^{n/2})$-lower bound for distributions over $[m]^n$, and described algorithms with $\tilde{O}(m^{n - 0.5})$-samples. The one-dimensional result's dependency on $\eps$ was improved by~\cite{CDGR18} and the optimal algorithm for the low-dimensional regime was given by Acharya, Daskalakis and Kamath~\cite{ADK15} who gave a tester with sample-complexity $O\left(\frac{m^{n/2}}{\eps^2} + \frac{1}{\eps^2}\left(\frac{n\log m}{\eps^2}\right)^n\right)$. For the high-dimensional regime (which is of interest of this paper) of distributions over the hypercube $\{-1,+1\}^n$, one can get stronger lower bounds than ones found by reduction to uniformity testing: Aliakbarpour, Gouleakis, Peebles, Rubinfeld and Yodpinayee~\cite{AGPRY19} prove a lower bound of $2^{(1 - \Theta(\sqrt{\eps}) - o(1))n}$ on the sample complexity. The best upper bound is currently at $\smash{2^{n} / 2^{\Theta_{\eps}(n^{1/5})}}$ samples due to Rubinfeld and Vasilyan~\cite{RV20}.

\paragraph{Distribution Testing Beyond Sample-Complexity.} Many distribution testing problems over high-dimensional domains incur sample complexities which are exponential in the dimension. As a result, various works have sought models and techniques to overcome these lower bounds, which can be divided between those which assume structure on the input, and those which provide stronger access. Works assuming additional structure on the input include monotonicity~\cite{RS09}, low-degree Bayesian networks~\cite{CDKS17, DP17, ABDK18, DKP23}, Markov random fields~\cite{GLP18, DDK19, BBCSV20}, ``flat'' histogram structure~\cite{DKP19}, or structured truncations~\cite{DNS23, DLNS24}. On the other hand, the subcube conditional model follows the other approach on assuming stronger access.~\cite{BC18} was the first to obtain polynomial query complexities for various testing problems; for uniformity testing over $\{-1,1\}^n$,~\cite{CCKLW21} showed the complexity is $\tilde{\Theta}(\sqrt{n}/\eps^2)$~and~\cite{CM24} extended it to hypergrids. In this work, we use an approach of~\cite{CJLW21b} which studied subcube conditioning for testing and learning $k$-junta distributions (those which have at most $k$ non-uniform variables). Other accesses include (unrestricted) conditioning on the domain~\cite{CRS15,CFGM16} (see also, improvements~\cite{FJOPS15, ACK15, KT18, N21, CCK24, CCKM24}),  queries to the probability density function or cumulative distribution function~\cite{BDKR05, CR14}, conditioning on prefixes for hidden Markov models~\cite{MKKZ23}, and samples which reveal their probability~\cite{OS18}. 


\section{Testing Monotonicity}

In this section, we show that using the directed and real-valued version of Talagrand's inequality, we may design an ``edge tester'' for testing monotonicity of distributions using subcube conditioning. In particular, we give the following theorem.
\begin{theorem}\label{thm:mon-ub}
    There exists an algorithm that receives as input subcube conditioning access to an unknown distribution $p$ supported on $\{-1,1\}^n$, as well as an accuracy parameter $\eps$. The algorithm makes $\tilde{O}(n/\eps^2)$ subcube conditioning queries and satisfies the following guarantees:
    \begin{itemize}
        \item If $p$ is monotone, the algorithm outputs ``accept'' with probability at least $0.9$.
        \item If $p$ is $\eps$-far from monotone, the algorithm outputs ``reject'' with probability at least $0.9$.
    \end{itemize}
\end{theorem}

The algorithm referred to in Theorem~\ref{thm:mon-ub} is given in Figure~\ref{fig:alg}.
We break up the proof of \Thm{mon-ub}
into a few parts. First, we argue about the running time.

\begin{figure}
\begin{framed}

\textbf{Algorithm for Testing Monotonicity of Distributions}. We receive as input subcube conditioning access to an unknown distribution $p$ which is supported on $\{-1,1\}^n$. Furthermore, we receive the accuracy parameter $\eps \in (0,1)$. We let $c_0$ denote a sufficiently small constant. 
\begin{enumerate}
    \item For all integers $w\geq0$ such that $2^w = \tilde{O}(n/\eps^2)$, repeat the following $t = O(2^{w} \log(n/\eps))$ times:
    \begin{itemize}
        \item Sample $\bx \sim p$ and $\bi \sim [n]$, and consider the restriction $\brho \in \{-1,1,*\}^n$ given by $\brho_j = \bx_j$ if $j \neq \bi$, and $\brho_{\bi} = *$. 
        \item Let $\eta = c_0^2\eps^2 \cdot 2^{w} / (16n \cdot \log(n/\eps) \cdot \log n)$ and take $m = O(\log (n/\eps)/\eta)$ subcube conditioning queries with restriction $\brho$ while counting the number of $1$'s and $-1$'s in coordinate $\bi$ observed. If the number of $-1$'s observed is larger than $m\left(1/2 + \sqrt{\eta}/2\right)$, output ``reject.''
    \end{itemize}
    \item If the algorithm has not rejected, output ``accept.''
\end{enumerate}

\end{framed}
\caption{Algorithm for Testing Monotonicity of Distributions} \label{fig:alg}
\end{figure}

\begin{claim}
    The query complexity is $\tilde{O}(n / \eps^2)$.
\end{claim}
\begin{proof}
    We simply upper bound the query complexity by inspecting Figure~\ref{fig:alg}. We have that (disregarding constant factors) the query complexity is the sum over all integers $w \geq 0$ such that $2^{w} = \tilde{O}(n/\eps^2)$ of 
    \[ O(2^{w} \cdot \log(n/\eps)) \cdot O\left(\dfrac{n\cdot \log^2(n/\eps) \cdot \log n}{c_0^2 \eps^2 \cdot 2^{w}} \right) = \tilde{O}(n / \eps^2). \]
    There are $O(\log(n/\eps))$ such settings of $w$, so the total complexity is still $\tilde{O}(n / \eps^2)$.
\end{proof}

\begin{lemma}
    Whenever $p$ is monotone, the algorithm outputs ``accept'' with probability at least $0.9$.
\end{lemma}

\begin{proof}
    Note that if $p$ is monotone, then for any restriction $\rho \in \{-1, 1, *\}^n$, which has one coordinate $i$ with $\rho_i = *$, a sample $\by$ from $p_{|\rho}$ must have the probability that $\by_i$ is 1 is at least the probability that it is $-1$. A standard Hoeffding bound implies that if one takes $m = O(\log(n/\eps)/\eta)$ samples of some event which is more likely to be $1$ than $-1$, the probability that the number of $1$'s observed is smaller than $m (1/2 - \sqrt{\eta}/2)$ is smaller than $\poly(\eps / n)$, for an arbitrarily large polynomial. Note that the number of times we may wrongfully reject is at most the query complexity, which is at most $\tilde{O}(n/\eps^2)$. So we may union bound as desired.
\end{proof}

\begin{lemma}\label{lem:far-case-reject}
    Whenever $p$ is $\eps$-far from monotone, the algorithm outputs ``reject'' with probability at least $0.9$.
\end{lemma}

\begin{proof}
    We show that whenever $p$ is $\eps$-far from monotone, there exists some $\gamma \in \{0, \dots, h\}$ with $h = O(\log(n/\eps))$ and a setting of $\ell \in \{ 0, \dots, r\}$ where $r = O(\log(n/\eps))$ which satisfies $2^{2\gamma + r +1} = \tilde{O}(n / \eps^2)$ and
    \begin{align*}
        \Prx_{ \substack{\bx \sim p \\ \bi \sim [n]}}\left[\left(\dfrac{(p(\bx^{(\bi\to-1)}) - p(\bx^{(\bi\to1)}))^+}{p(\bx^{(\bi\to-1)}) + p(\bx^{(\bi\to1)})}\right)^2 \geq \eta\right] \geq \frac{1}{r \cdot 2^{\gamma + \ell}}.
    \end{align*}
    for $\eta = c_0^2 \eps^2 \cdot 2^{2\gamma+\ell} / (16 h \cdot n \cdot \log n)$. When the algorithm iterates over all $w \geq 0$ such that $2^w = \tilde{O}(n/\eps^2)$, it will eventually consider $w = 2\gamma + \ell$.
    This implies that except with probability $0.01$, one of the $t$ samples $\bx \sim p$ and $\bi \sim [n]$ satisfy the above bound, since we repeat $t = O(2^{w}\log(n/\eps))$ times and $2^{w} = 2^{2\gamma+\ell}$ is larger than $2^{\gamma+\ell}$. Once that is set, with probability except $0.01$, the algorithm outputs reject; the subcube conditioning query $\brho$ is exactly sampling from $\{-1,1\}$ whose probability of being $-1$ is at least $1/2 + \sqrt{\eta}$. By a Hoeffding bound, the probability that the number of $-1$'s is smaller than $m(1/2 + \sqrt{\eta}/2)$ is at most $0.01$. From Corollary~\ref{cor:l1-tal}, for 
    small enough constant $c_0$. the fact that $p$ is $\eps$-far from monotone implies,
    \begin{align*}
        \frac{c_0\eps}{\sqrt{\log n}} &\leq \sum_{x \in \{-1,1\}^n} \sqrt{\sum_{i:x_i = -1} \left( \left(p(x^{(i\to-1)}) - p(x^{(i\to1)} \right)^+ \right)^2} \\
        &= \Ex_{\bx \sim p}\left[ \sqrt{\sum_{i:\bx_i=-1} \left(\dfrac{(p(\bx^{(i\to-1)}) - p(\bx^{(i\to1)}))^+}{p(\bx^{(i\to-1)})} \right)^2} \right].
    \end{align*}
    Furthermore, $p(\bx^{(i\to-1)}) - p(\bx^{(i\to1)}) \geq 0$ implies $p(\bx^{(i\to-1)}) + p(\bx^{(i\to 1)}) \leq 2 p(\bx^{(i\to-1)})$. So we may lower bound
    \begin{align}
        \frac{c_0 \eps}{4\sqrt{\log n}} \leq \Ex_{\bx \sim p}\left[ \sqrt{\sum_{i:\bx_i=-1} \left(\dfrac{(p(\bx^{(i\to-1)}) - p(\bx^{(i\to1)}))^+}{p(\bx^{(i\to-1)}) + p(\bx^{(i\to1)})} \right)^2} \right] \label{eq:exp1}
    \end{align}
    Notice that the maximum quantity within the expectation in (\ref{eq:exp1}) is $\sqrt{n}$, since each of the terms being added is between $0$ and $1$. Therefore, there must exist some $\gamma \in \{0,\dots, h \}$ with $h = \lceil \log_2(4\sqrt{n}/(c_0\eps))\rceil + 1 = O(\log(n/\eps))$ which satisfies
    \begin{align}
     \Prx_{\bx \sim p}\left[\sum_{i:\bx_i=-1} \left(\dfrac{(p(\bx^{(i\to-1)}) - p(\bx^{(i\to1)}))^+}{p(\bx^{(i\to-1)}) + p(\bx^{(i\to1)})} \right)^2 \geq \frac{c_0^2 \cdot \eps^2 \cdot 2^{2\gamma}}{16h\log n} \right] \geq \frac{1}{2^{\gamma}}. \label{eq:good-1}
    \end{align}
    Thus, consider any one of those values of $x$, and in order to simplify the notation, we define 
    \[ \xi \eqdef \frac{c_0^2 \cdot \eps^2 \cdot  2^{2\gamma}}{16h\log n} \qquad \nu_i = \left(\dfrac{(p(x^{(i\to-1)}) - p(x^{(i\to1)}))^+}{p(x^{(i\to-1)}) + p(x^{(i\to1)})} \right)^2,\]
    so that we assume to fix $x$ such that $\sum_{i=1}^n \nu_i \geq \xi$, and each $\nu_i \in [0, 1]$. Consider a partition of the coordinates of $[n]$ into groups $B_1, \dots, B_{r}$, such that $i \in B_{\ell}$ whenever the $i$-th coordinate contributes between $\xi 2^{\ell}/ n$ and $\xi 2^{\ell+1} / n$, and $r$ is chosen is the value $\xi 2^{r+1} / n$ is between $1$ and $2$ (note that, since $\nu_i \in [0, 1]$, $B_{r'}$ for $r' > r$ must be empty), so $r = O(\log(n/\xi))$. Then, there must be some $\ell$ with $|B_{\ell}| \geq n / (r \cdot 2^{\ell+1})$, and this implies
    \begin{align}
        \Prx_{\bi \sim [n]}\left[\left(\dfrac{(p(\bx^{(\bi\to-1)}) - p(\bx^{(\bi\to1)}))^+}{p(\bx^{(\bi\to-1)}) + p(\bx^{(\bi\to1)})} \right)^2 \geq \frac{\xi \cdot 2^{\ell}}{n}  \right] \geq \frac{1}{r \cdot 2^{\ell}}.\label{eq:good-2}
    \end{align}
    The desired bound then follows from the setting of $\eta$, and lower bounding the probability that $\bx \sim p$ satisfies the event of (\ref{eq:good-1}), and then $\bi \sim [n]$ satisfies the event of (\ref{eq:good-2}).
\end{proof}

\subsection{The Real-Valued Directed Talagrand Inequality} \label{sec:tal}

We will prove a ``directed isoperimetric theorem" for real-valued functions. This is an important
tool used for the analysis of the monotonicity tester. We define notions of the
directed boundary for Boolean functions. 

Let $f:\{-1,1\}^n \to [0,1]$ be a function defined on the $n$-dimensional hypercube. The $L_1$-distance 
of $f$ from monotonicity is defined as
\begin{equation*}
	\dist_1(f) \eqdef \min_{g~:~\text{monotone}} ~~\Ex_{\bx \sim \{-1,1\}^n}\left[ |f(\bx) - g(\bx)| \right]
\end{equation*}
where the expectation is over the uniform distribution over $\{-1,1\}^n$.
For a point $x\in \{-1,1\}^n$, define the directed derivative $\grad^-f(x)$ to be the $n$-dimensional vector defined as 
\begin{equation}\label{eq:defgrad}
	\left(\grad^-f(x)\right)_i \eqdef \begin{cases}
		0 & \textrm{if $x_i = 1$} \\
		\left(f(x) - f(x+2e_i)\right)^+ & \text{otherwise}
	\end{cases}
\end{equation}
where $(z)^+$ is a shorthand for $\max(z,0)$. For Boolean-valued $f:\{-1,1\}^n \to \{0,1\}$, the distance $\dist_1(f)$ corresponds to the ``normal'' Hamming distance notion, $\dist_0(f)$.
Based on isoperimetric theorems of Talagrand~\cite{Tal93}, the quantity $\Exp_{\bx} \norm{\grad^-f (\bx)}_2$ can be thought
of as a ``directed surface area" for the function $f$. A deep isoperimetric theorem of Khot, Minzer, and Safra~\cite{KMS18} (see, also~\cite{PRW22}, who showed how to remove the final logarithmic factor)
lower bounds this surface area by the distance to monotonicity.

\begin{theorem}[\cite{KMS18, PRW22}]\label{thm:booliso}
	There exists a universal constant $C > 0$ such that for every $f \colon \{-1,1\}^n \to \{0,1\}$,
$\Exp_{\bx} \norm{\grad^-f (\bx)}_2 \geq C\cdot \dist_0(f)$.
\end{theorem}

Theorem~\ref{thm:l1-talagrand} gives a real-valued generalization of the above theorem, with a $\sqrt{\log n}$ loss in the bound. The proof appears in Subsection~\ref{sec:proof-tal}, but we state the following corollary used in the tester's analysis.
\begin{corollary} \label{cor:l1-tal} Let $p$ be a distribution over $\{-1,1\}^n$ that is $\eps$-far
from monotone. Then 
\[
\sum_{x \in \{-1,1\}^n} \sqrt{\sum_{i:x_i = -1} \left( \left(p(x^{(i\to-1)}) - p(x^{(i\to1)} \right)^+ \right)^2} = \Omega\left(\frac{\eps}{\sqrt{\log n}}\right).
\]
\end{corollary}

\begin{proof} Let $\eps(p)$ be the distance of $p$ to monotonicity. Note that this is the distance
over distributions, while \Thm{l1-talagrand} refers to $L_1$-distance between arbitrary functions.
So we need an extra calculation to apply \Thm{l1-talagrand}.

Let $\cM$ be the set of monotone distributions. Then, $\eps(p) = \min_{q \in \cM} \dtv(p,q)
= \min_{q \in \cM} \|p-q\|_1/2$. On the other hand, $\dist_1(p) = \min_{g: \textrm{monotone}} \Exp_{\bx}|p(\bx) - g(\bx)|
= 2^{-n} \min_{g: \textrm{monotone}} \|p-g\|_1$. Note that the minimizer $g$ is non-negative, since $p$
is non-negative. Hence, the function $f = g/\|g\|_1$ is a distribution. 

By triangle inequality,
\begin{eqnarray*}
\eps(p) \leq \|p - f\|_1 \leq \|p-g\|_1 + \|f-g\|_1 = \|p-g\|_1 + \|g - g/\|g\|_1 \|_1
\end{eqnarray*}
Observe that $\|g - g/\|g\|_1 \|_1 = \sum_x |g(x) - g(x)/\|g\|_1| = |1 - 1/\|g\|_1| \cdot \sum_x |g(x)|
= | 1 - \|g\|_1|$. Since $p$ is a distribution, this expression is equal to $| \|p\|_1 - \|g\|_1|$.
And finally, $| \sum_x (|p(x)| - |g(x)|)| \leq \sum_x |p(x) - g(x)| = \|p-g\|_1$. 
Overall, we deduce that $\eps(p) \leq 2\|p-g\|_1$. Recall that $\dist_1(p)$ is defined
using an expectation over the domain, so $\eps(p) \leq 2 \cdot 2^n \dist_1(p)$. 

With our lower bound for $\dist_1(g)$, we can apply \Thm{l1-talagrand}. 
So $\Exp_{\bx} \norm{\grad^-p(\bx)} = \Omega(\dist_1(p)/\sqrt{\log n}) = \Omega(2^{-n} \eps(p)/\sqrt{\log n})$.
We expand out the expression for $\grad^-p(x)$ to wrap up the proof.
\begin{eqnarray*}
\Ex_{\bx\sim\{-1,1\}^n}\left[ \norm{\grad^-p(\bx)} \right] = 2^{-n} \sum_{x \in \{-1,1\}^n} \norm{\grad^-p(x)}
= 2^{-n}  \sum_{x \in \{-1,1\}^n} \sqrt{\sum_{i:x_i = -1} \left( (p(x) - p(x+2e_i) )^+ \right)^2}
\end{eqnarray*}
As argued above, this expression is lower bounded by $\Omega(2^{-n} \eps(p)/\sqrt{\log n})$.
The $2^{-n}$ terms ``cancel out", and noting that $\eps(p) \geq \eps$, we get the desired bound.
\end{proof}

\subsection{The proof of \Thm{l1-talagrand}} \label{sec:proof-tal}

By a simple translation and rescaling argument, we reduce the function range to $[0,1]$.
This will make subsequent calculations easier.

\begin{claim} \label{clm:rescale} Consider $f:\{-1,1\}^n \to \R$. For positive $\alpha \in \R^+$ and
any $\beta \in \R$, define the function $\hat{f}$ where $\hat{f}(x) = \alpha f(x) + \beta$. Then,
$\EX_x[\|\nabla^- \hat{f}\|_2]/\dist_1(\hat{f}) =\EX_x[\|\nabla^- f\|_2]/\dist_1(f) $.
\end{claim}

\begin{proof} The monotonicity violations in $f$ and $\hat{f}$ are identical.
For any point $x$ and coordinate $i$, $ (\nabla^- \hat{f}(x))_i = \alpha (\nabla^- f(x))_i$.
Hence, $\EX_x[\|\nabla^- \hat{f}\|_2] = \alpha \EX_x[\|\nabla^- f\|_2]$.
For a function $g$, let $\alpha g + \beta$ be the function whose value at $x$
is $\alpha g(x) + \beta$.
\begin{align*}
 \dist_1(\hat{f}) = \min_{\hat{g}: \textrm{monotone}} \|\hat{f}-\hat{g}\|_1 &= \min_{\hat{g}: \textrm{monotone}} \| (\alpha f + \beta) - \hat{g}\|_1
\min_{\hat{g}: \textrm{monotone}} \|(\alpha f + \beta) - (\alpha (\alpha^{-1}(\hat{g} - \beta)) + \beta)\|
\end{align*}
Monotonicity is preserved by positive scaling and translation, so $\hat{g}$ is monotone iff $(\alpha^{-1}(\hat{g} - \beta))$
is monotone.
Hence,
\begin{align*}
 \dist_1(\hat{f}) = \min_{g: \textrm{monotone}} \|(\alpha f + \beta) - (\alpha g + \beta)\|_1
= \min_{g: \textrm{monotone}} \|\alpha f - \alpha g\|_1 = \alpha \ \dist_1(f)
\end{align*}
We conclude that  
$\EX_x[\nabla^- \hat{f}\|_2]/\dist_1(\hat{f}) =\EX_x[\|\nabla^- f\|_2]/\dist_1(f) $.
\end{proof}

Given $f$, we technically work with the function $\hat{f} = f/2M + 1/2$,
where $M = \max_x |f(x)|$. Observe that $\hat{f}$ has range in $[0,1]$,
and by \Clm{rescale}, the statement of \Thm{l1-talagrand} for $\hat{f}$ implies
the statement for $f$.

Abusing notation, we just assume that $f:\{-1,1\}^n \to [0,1]$.
We use the technique of Berman, Raskhodnikova, and Yaroslavtsev~\cite{BeRaYa14} of using threshold Boolean functions to relate the real-valued $f$ to Boolean functions. 

\noindent
Given $t\in [0,1]$ consider the following Boolean function (Definition 2.1 in~\cite{BeRaYa14}) $f_t : \{-1,1\}^n \to \{0,1\}$
\[
	f_t(x) = \begin{cases}
		1 & \text{if}~ f(x) \geq t; \\
		0 & \text{if}~ f(x) < t
	\end{cases}
\]
\noindent
It is easy to see that for any $x\in \{-1,1\}^n$,
\[
	f(x) = \int_0^{f(x)} dt  =  \int_0^1 f_t(x)~dt ~= \Ex_{\bt \sim [0,1]} \left[ f_{\bt}(x) \right]
\]
where the expectation is over $t$ uniformly distributed over $[0,1]$.
One can perform analogous calculations to relate the $L_1$ distance of (the real valued)
$f$ to the $L_0$ distance of (the Boolean) $f_t$s.  

\begin{lemma}[Lemma 2.1~\cite{BeRaYa14}]\label{lem:bry}
	\noindent
	For any function $f:\{-1,1\}^n \to [0,1]$, $\dist_1(f) = \int_0^1 \dist_0(f_t) dt = \Exp_{\bt} \left[ \dist_0(f_{\bt})\right]$.
\end{lemma}

The main work is in relating the (directed) gradients of $f$ to the corresponding
gradients of $f_t$. This is where we suffer a $\sqrt{\log n}$ loss.
 
%
%
%

\begin{lemma}
	For all $x\in \{-1,1\}^n$, $\norm{\grad^- f(x)}_2 = \Omega(1/\sqrt{\log n}) \Exp_t \norm{\grad^- f_t(x)}_2 $.	
\end{lemma}
\begin{proof}
Fix any $x \in \{-1,1\}^n$. Let $y_1, \ldots, y_d \in \{-1,1\}^n$ denote the ``up''-neighbors of $x$ which satisfy $f(x) > f(y_j)$. In particular, there are at most $d \leq n$ points $y_1 ,\dots, y_d$ such that, for every $j \in [d]$, $y_j = x + 2e_i$ for some $i$, and in addition, $f(x) > f(y_j)$.
Order the indices so as to assume $f(y_1) \leq f(y_2) \leq \cdots \leq f(y_d)$ and let $a_j := f(x) - f(y_j)$ (and so $a_1 \geq a_2 \geq \cdots \geq a_d$). By definition, we have defined $a_1,\dots, a_d$ to have $\norm{\grad^- f(x)}_2 = (\sum_{j=1}^d a^2_j)^{1/2}$.

For $t \in [0, 1]$, consider the function $f_t$, and let edge $(x,y_j)$ be called a violation in $f_t$ if $f_t(x) = 1$ and $f_t(y_j) = 0$.
Observe that only violated edges contribute to $\norm{\grad^-f_t(x)}$. 
Notice that for any $t \in (f(y_i), f(y_{i+1})]$, the edge $(x,y_j)$ is a violation in $f_t$ iff $j \leq i$.
Hence, if $t \in (f(y_i), f(y_{i+1})]$, then the vector $\grad^-f_t(x)$ 
has exactly $i$ non-zeros and $\norm{\grad^- f_t(x)}_2 = \sqrt{i}$. For $i < d$, the probability that $t \in (f(y_i), f(y_{i+1})]$
is exactly $y_{i+1} - y_i = a_i - a_{i+1}$. The probability that $t \in (f(y_d), x]$
is exactly $a_d$.
%
Thus,
\[
	\Ex_{\bt\sim[0,1]}\left[ 	\norm{\grad^- f_{\bt}(x)}_2 \right] = \sum_{i=1}^{d-1} \left(a_i - a_{i+1}\right)\sqrt{i} + a_d \sqrt{d} = \sum_{i=1}^d a_i \cdot \left(\sqrt{i} - \sqrt{i-1}\right)
\]
By Cauchy-Schwarz and the following calculation, we complete the proof
\begin{equation*}
	\Ex_{\bt\sim [0,1]}\left[ 	\norm{\grad^- f_{\bt}(x)}_2 \right]\leq \norm{\grad^- f(x)}_2 \cdot \sqrt{\sum_{i=1}^d \left(\sqrt{i} - \sqrt{i-1}\right)^2}\leq O(\sqrt{\log n})\cdot \norm{\grad^- f(x)}_2
\end{equation*}
since $\sqrt{i} - \sqrt{i-1} = 1/(\sqrt{i}+\sqrt{i-1}) \leq 1/\sqrt{i}$, and so $\sum_{i \leq d} (\sqrt{i} - \sqrt{i-1})^2 \leq \sum_{i \leq d} 1/i = O(\log d)$.
%
\end{proof}

We now complete the proof of \Thm{l1-talagrand}. By the above lemma,
\begin{align*}
\Ex_{\bx\sim\{-1,1\}^n}\left[ \norm{\grad^- f(\bx)}_2\right] &= \Omega(1/\sqrt{\log n}) \Ex_{\substack{\bx \sim \{-1,1\}^n \\ \bt \sim [0,1]}} \left[ \norm{\grad^- f_{\bt}(\bx)}_2\right] \\
    &= \Omega(1/\sqrt{\log n}) \Ex_{\bt \sim [0,1]} \left[  \Ex_{\bx\sim\{-1,1\}^n}\left[ \norm{\grad^- f_{\bt}(\bx)}_2 \right] \right].
\end{align*}
By the directed Boolean isoperimetric statement of \Thm{booliso}, $\Exp_{\bx} \norm{\grad^- f_t(\bx)}_2 = \Omega(\dist_0(f_t))$.
We apply this bound and then \Lem{bry} to relate back to $f$.
\begin{align*}
\Ex_{\bx\sim\{-1,1\}^n}\left[ \norm{\grad^- f(\bx)}_2 \right] = \Omega(1/\sqrt{\log n}) \Ex_{\bt\sim[0,1]} \left[\dist_0(f_{\bt})\right] = \Omega(1/\sqrt{\log n}) \cdot \dist_1(f)
\end{align*}


\section{Lower Bound for Testing Monotonicity}\label{sec:mon-lb}

In this section, we prove a query complexity lower bound on testing monotonicity of distributions using subcube conditioning queries.
\begin{theorem}[Monotonicity Testing -- Lower Bound]\label{thm:mon-lb}
For any $\eps \in (0,1)$, an $\eps$-test for monotonicity of distributions must make $\tilde{\Omega}(n/\eps^2)$ queries.
\end{theorem}

\subsection{Preliminaries} \label{sec:prelims}

Our lower bound proofs proceed by Yao's method. We consider a property of distribution $\calP$ that we want to test.
We describe two distributions, $\Dyes$ and $\Dno$, supported on \emph{product} distributions over $\{-1,1\}^n$ which, in addition, satisfy the following properties:
\begin{itemize}
\item Every distribution $p$ in $\Dyes$ is a distribution over $\{-1,1\}^n$ lies in $\calP$.
\item A distribution $\bp$ drawn from $\Dno$ is a distribution over $\{-1,1\}^n$ which is $\eps$-far from $\calP$ with probability at least $0.99$ (over the draw of $\bp \sim \Dno$).   
\end{itemize}
Consider any deterministic algorithm which can $\eps$-test monotonicity which makes $q$ queries. The algorithm is specified by a depth-$q$ decision tree of subcube conditioning queries; each non-leaf node of the decision tree specifies a subcube $\rho \in \{-1,1, *\}^n$ which is the query performed at that node, and has $2^{\stars(\rho)}$ many children, corresponding to the possible completions of $\rho$ that the algorithm would receive after such a query; each leaf is labeled ``accept'', or ``reject'' corresponding to the output. In the case input distributions $p$ are promised to be \emph{product} distributions, algorithms with subcube query access significantly simplify, allowing us to prove lower bounds against the \emph{sample complexity} of $\eps$-testing algorithms.

\begin{lemma}[Decision Tree to i.i.d Samples]\label{lem:tree-to-iid}
Let $p$ be a product distribution supported on $\{-1,1\}^n$, and $\Alg$ be a deterministic $q$-query algorithm with subcube conditioning access. There exists a function $\Alg' \colon \{-1,1\}^{nq} \to \{ \text{``accept''}, \text{``reject''}\}$ which exactly simulates the algorithm on independent samples, i.e.,
\begin{align*}
\Prx\left[ \Alg(p) = \text{``accept''} \right] = \Prx_{\bx_1,\dots, \bx_q \sim p}\left[ \Alg'(\bx_1,\dots, \bx_q) = \text{``accept''}\right].
\end{align*}
\end{lemma}\begin{proof}
We describe the function $\Alg' \colon \{-1,1\}^{nq} \to \{\text{``accept''}, \text{``reject''} \}$ by specifying how it computes on independent samples $\bx_1, \dots, \bx_q$ using the depth-$q$ decision tree $\Alg$. The algorithm maintains the current node $u$, initially set to the root of $\Alg$, and a counter $c$, initially set to $1$. If $u$ is a leaf, then we output either ``accept'' or ``reject'', according to the decision held at $u$. Otherwise, $u$ contains a string $\rho \in \{-1,1, *\}^n$ which specifies a subcube conditioning query. We update $u$ to be the child corresponding to the completion $(\bx_c)_{|\stars(\rho)} \in \{-1,1\}^{\stars(\rho)}$ of the sample $\bx_c$ and we increment $c$. We note that, since the distribution $p$ is a product distribution, $(\bx_c)_{|\stars(\rho)}$ is distributed as a draw from $p_{|\rho}$, and since the counter is incremented, all the used queries are mutually independent. Hence, this is a perfect simulation of subcube conditioning queries with i.i.d samples when the input distribution $p$ is product.
\end{proof}

In the following subsection, we describe the two distribution $\Dyes$ and $\Dno$; we prove these are monotone and far-from monotone distributions in Lemma~\ref{lem:mon-dist} and Lemma~\ref{lem:far-mon-dist}, respectively. 
As per Lemma~\ref{lem:tree-to-iid}, in order to prove a $q$-query lower bound, it suffices to show that for any function $\Alg \colon \{-1,1\}^{nq} \to \{\text{``accept''}, \text{``reject''} \}$, 
\begin{align}
\left| \Prx_{\substack{\bp \sim \Dyes \\ \bx_1,\dots, \bx_q \sim \bp}}\big[ \Alg(\bx_1,\dots, \bx_q) = \text{``accept''}\big] - \Prx_{\substack{\bp \sim \Dno \\ \bx_1,\dots, \bx_q \sim \bp}}\big[ \Alg(\bx_1,\dots, \bx_q) = \text{``accept''}\big] \right| = o(1). \label{eq:indistinguish}
\end{align}
In Section~\ref{sec:indist}, we prove that (\ref{eq:indistinguish}) holds for any $q$-sample algorithm $\Alg$ with $q$ smaller than $n / (\eps^2\cdot \polylog(n))$. This concludes the proof of Theorem~\ref{thm:mon-lb}.

\subsection{The one-dimensional mean distributions} \label{sec:mean-dist}

A distribution supported on product distributions on $\{-1,1\}^n$ may be equivalently specified by describing a distribution on vectors in $[-1,1]^n$, corresponding to the mean vectors of the distribution. 
In this section, we show the existence of one-dimensional distributions with special properties. In the next subsection, we use these
distributions to generate mean vectors, from which the hypercube distributions are constructed. The properties of the next
lemma are crucial for the main lower bound.

\begin{lemma}[One-Dimensional Bias Distributions]\label{lem:one-dim-dist}
Let $c_0$ be some absolute positive constant.
Fix any natural number parameter $\ell$. There exist two discrete distributions $\cA$ and $\cB$ with the following properties.
\begin{asparaitem}
	\item $\cA$ is supported on $\{0\} \cup [\ell^3]$ and $\cB$ is supported on $\{-1, 0\} \cup [\ell^3]$.
	\item For every $k \leq \ell$, $\EX_{i \sim \cA}[i^k] = \EX_{j \sim \cB}[j^k]$.
	\item $\Pr_{j \sim \cB}[j = -1] > c_0$.
\end{asparaitem}
\end{lemma}

The first step of the proof is to construct a solution to a linear system that captures the second point above (the equality of moments).
For convenience, let $\alpha_j= j^3$ for each $j\in [\ell]$.
Consider the following $(\ell+1) \times (\ell+1)$ matrix $A$:
\begin{align*}
A := \left[ \begin{array}{ccccc} 1 & 0 & 0 & \dots & 0 \\ 
				       1 & 1 & 1 & \dots & 1 \\
				       -1 & \alpha_1^1 & \alpha_2^1 & \dots & \alpha_{\ell}^1 \\
				       (-1)^2 & \alpha_1^2 & \alpha_2^2 & \dots & \alpha_{\ell}^2 \\
				       \vdots & \vdots & \vdots & \vdots & \vdots \\
				       (-1)^{\ell-1} & \alpha_1^{\ell-1} & \alpha_2^{\ell-1} & \dots & \alpha_{\ell}^{\ell-1} \end{array} \right].
\end{align*}

\begin{claim}\label{clm:inv-and-z}
The matrix $A$ 
is invertible and the vector $z \in \R^{\ell+1}$ satisfying $A z = e_1$ has $\|z\|_1 = O(1)$.
\end{claim}

\begin{proof} It is convenient to index the rows/columns of $A$ with $0,1,2,\ldots,\ell$.
The proof of Claim~\ref{clm:inv-and-z} mostly follows similar calculations  in the proof of Claim~6.3 of~\cite{CJLW21b}. Let $V$ denote the transpose of the $\ell \times \ell$ Vandermonde matrix on $\alpha_1,\dots, \alpha_{\ell}$, so we may substitute for the determinant
\[ \det(A ) = \det(V ) = \prod_{\substack{i,j \in [\ell] \\ i < j}} (\alpha_j - \alpha_i) \neq  0,\]
as long as $\alpha_1,\dots, \alpha_{\ell}$ are distinct, and this means that $A$ is invertible. 

We now compute $z$ using Cramer's rule: $z_0=1$ and for each $i\in [\ell]$, $z_i$ is given by
\begin{align*}
z_i = \dfrac{\det(A_i)}{\det(A)} = (\pm 1)\cdot \dfrac{\det(V(-1; \alpha_{-i}))}{\det(V )}.
\end{align*}
The numerator $\det(A_i)$ denotes the determinant of the matrix where the $i$-th column of $A$ is replaced by $e_1$. Then, $\det(A_i)$ is the determinant of the Vandermonde matrix of the entries $-1$, and well as $\alpha_1,\dots, \alpha_{\ell}$ except for $\alpha_i$, which we refer  to as $V(-1;\alpha_{-i})$ above. By the formula of the determinant for Vandermonde matrices, we have 
\begin{align*}
\det\big(V(-1; \alpha_{-i})\big) &= \prod_{j \in [\ell] \setminus \{i \}} (\alpha_j + 1) \cdot \prod_{\substack{j_1, j_2 \in [\ell] \setminus \{i\} \\ j_1 < j_2}} (\alpha_{j_2} - \alpha_{j_1}) \qquad\text{and}\\[0.6ex]
\det(V ) &= \prod_{j > i} (\alpha_{j} - \alpha_i) \prod_{j < i} (\alpha_i - \alpha_j) \cdot \prod_{\substack{j_1, j_2 \in [\ell] \setminus \{i\} \\ j_1 < j_2}} (\alpha_{j_2} - \alpha_{j_1}),
\end{align*}
and $|z_i|$ is given by 
\begin{align*}
|z_i| = \prod_{j \in [\ell] \setminus \{i \}} \dfrac{\alpha_j + 1}{|\alpha_j - \alpha_i|} = \prod_{j \in [\ell] \setminus \{i\}} \frac{\alpha_j}{|\alpha_j - \alpha_i|} \left( 1 + \frac{1}{\alpha_j} \right) = \prod_{j \in [\ell] \setminus \{i\}} \frac{\alpha_j}{\alpha_j - \alpha_i} \prod_{j \in [\ell] \setminus \{i\}} \left( 1 + \frac{1}{\alpha_j} \right),
\end{align*}
and therefore,
\begin{align*}
    |z_i| \leq \left| \prod_{j \in [\ell]\setminus\{i\}} \frac{\alpha_j}{\alpha_j - \alpha_i} \right| \cdot \exp\left( \sum_{j=1}^{\ell} \frac{1}{j^3} \right) \leq O(1) \cdot \left| \prod_{j \in [\ell]\setminus\{i\}} \frac{\alpha_j}{\alpha_j - \alpha_i} \right|.
\end{align*}
Finally, the bound on the $\|z\|_1$ follows from summing up the right-most expression above, which is $O(1)$, with the exact derivation in Claim~6.3 of~\cite{CJLW21b}.
\end{proof}

\begin{proof} (of Lemma~\ref{lem:one-dim-dist}) Consider the vector $z \in \R^{\ell+1}$ from \Clm{inv-and-z}.
For convenience we index $z$ using $0, \ldots,\ell$. Note that $z_0 = 1$. Let $N \subset [\ell]$ be the 
set of coordinates that are negative, and $P \subseteq [\ell]$ be the set of positive coordinate. (We do 
not put the index $0$ in $P$, since we treat $z_0 = 1$ separately.

The distribution $\cA$ is supported on $\{i^3 | i \in N\}$. The probability of $i^3$ is $z_i/\|z\|_1$,
and the probability of $0$ is $1 - \sum_{i \in N} z_i/\|z\|_1$.

The distribution $\cB$ is supported on $\{-1\} \cup \{j^3 | j \in P\}$. The probability of $-1$ is $z_0/\|z\|_1$,
the probability of $j^3$ is $z_j/\|z\|_1$, and the probability of $0$ is the remainder $1 - z_0/\|z\|_1 - \sum_{j \in P} z_j/\|z\|_1$.

The first bullet of the lemma holds by the above construction. For the second lemma, consider the row of the matrix $A$
with the $k$th powers. Since $Az = 0$, that row leads to the equation $(-1)^k z_0 + \sum_{j \in P} \alpha^k_j z_j = \sum_{i \in N} \alpha^k_i z_i$.
Dividing by $\|z\|_1$, we deduce that $\EX_{j \sim \cB}[j^k] = \EX_{i \sim \cA} [i^k]$.
Finally, since $\|z\|_1 = O(1)$, $\Pr_{j \sim \cB}[j = -1] = z_0/\|z\|_1 = 1/\|z\|_1 = \Omega(1)$.
\end{proof}

\subsection{The distributions $\cD_{yes}$ and $\cD_{no}$} \label{sec:dist}

Given any ``mean" vector $\bmu = (\bmu_1, \bmu_2, \ldots, \bmu_n) \in [-1,1]^n$, we can define a product distribution
on $\{-1,1\}^n$ as follows. We set each coordinate $x_i$ to be $1$ with probability $(1+\bmu_i)/2$ and zero 
with probability $(1-\bmu_i)/2$. Note that $\EX[x_i] = \bmu_i$.

Recall the distributions $\cA$ and $\cB$ given in \Lem{one-dim-dist}. 
We set $\ell = \log n/\log\log n$.

We now define $\Dyes$ and $\Dno$, which are distributions over 
distributions on $\{-1,1\}^n$. For $\Dyes$,
we generate $n$ independent entries $(a_1, a_2, \ldots, a_n)$ from $\cA$. 
We applying a scaling to define the mean vector $\bmu$ where $\bmu_i = \eps a_i/\sqrt{n}$.
We then take the corresponding distribution over $\{-1,1\}^n$. This describes a ``draw"
from $\cD_{yes}$.
The distribution $\cD_{no}$ is generated analogously from $\cB$.

Observe that all the coordinates means for $\Dyes$ are non-negative, while
a constant fraction of the corresponding means for $\Dno$ are negative.

%
%
%
\begin{lemma}\label{lem:mon-dist}
Every distribution in the support of $\Dyes$ is monotone.  
\end{lemma}

\begin{proof} Consider $\cD \sim \Dyes$. Since the mean vector is non-negative, for each coordinate $x_i$,
$\Pr[x_i = 1] \geq \Pr[x_i = -1]$. Since $\cD$ is a product distribution, this means that the probability
of a point cannot decrease if we flip a $-1$ (coordinate) to a $1$. Hence, $\cD$ is monotone.
\end{proof}

\begin{lemma}\label{lem:far-mon-dist}
A distribution $\bp \sim \Dno$ is $\Omega(\eps)$-far from monotone with probability at least $0.99$.
\end{lemma}

\begin{proof}
Let $\bp \sim \Dno$ and let $\bmu = (\bmu_1,\dots, \bmu_n)$ be the corresponding mean vector, where $\bmu_i\sim \calD_n$ for all $i \in [n]$. 
Let $\bN \subset [n]$ denote the coordinates $i \in [n]$ where $\bmu_i = -\eps / \sqrt{n}$. From the third item of Lemma~\ref{lem:one-dim-dist} and standard concentration inequalities, \smash{$|\bN| \eqdef m$} has size at least $c_0 n / 2$ with probability $1-o(1)$. We assume this is the case and that $m$ is even (so $\bN$ is a large subset of coordinates where $\bmu_i$ is negative). We will lower bound 
the distance to monotonicity by showing that there exists a matching $\calM$ of pairs $(x, y)$ in $\{-1,1\}^n \times \{-1,1\}^n$ (which depends on $\bp$) where $x_i \leq y_i$ for every $i \in [n]$ and
\begin{align*}
\sum_{(x,y) \in \calM} ( \bp(x) - \bp(y) ) = \Omega(\eps).
\end{align*}
Assuming this bound, we prove that $\|\bp - g\|_1 = \Omega(\eps)$ for any monotone $g$, which
implies that $\dtv(\bp, g)\ge \Omega(\eps)$. Consider $(x,y) \in \calM$.
Suppose $\max(|\bp(x) - g(x)|, |\bp(y) - g(y)|) < (\bp(x) - \bp(y))/2$.
Then $g(x) > \bp(x) - (\bp(x) - \bp(y))/2 = (\bp(x) + \bp(y))/2$.
And $g(y) < \bp(y) + (\bp(x) - \bp(y))/2 = (\bp(x) + \bp(y))/2$. So $g(x) > g(y)$,
contradicting the fact that $g$ is monotone. Hence, $|\bp(x) - g(x)| + |\bp(y) - g(y)| \geq (\bp(x) - \bp(y))/2$.
Summing over all pairs $(x,y) \in \calM$, $\|\bp - g\|_1 \geq \sum_{(x,y) \in \calM} ( \bp(x) - \bp(y) ) = \Omega(\eps)$.

%
%
In order to describe the matching $\calM$, we first let $|z|$ denote the number of entries in $z \in \{-1,1\}^m$ which are set to $1$. We consider a bijection $\sigma$ which maps vectors $z \in \{-1,1\}^m$ with $|z| = m / 2 - r$ to $\sigma(z) \in \{-1,1\}^m$ with $|\sigma(z)| = m/2 + r$ for every $r \in [m/2]$ and satisfies $z \preceq \sigma(z)$. In other words, $\sigma$ maps the bottom-half of the hypercube $\{-1,1\}^m$ to a comparable element in the top half; the fact such matchings exist follows straight-forwardly from chain decompositions of $\{-1,1\}^m$. Let
\begin{align*}
\calM = \left\{ (x, x_{\ol{\bN}}\sigma(x_{\bN}) : x_{\bN} \in \{-1,1\}^m \text{ contains at most $m/2$ entries set to $1$}\right\},
\end{align*}
where we use the notation $x_{\ol{\bN}} \sigma(x_{\bN})$ to denote the string in $\{-1,1\}^n$ whose $i$-th coordinate is $x_i$ if $i \in \ol{\bN}$ and $\sigma(x_{\bN})_i$ if $i\in \bN$. The fact $z \preceq \sigma(z)$ and $\sigma$ is a bijection gives us the desired matching. Notice that, whenever $(x,y) \in \calM$ with $|x| = m/2-r$, we have
\begin{align*} 
\bp(y) &= \bp(x) \cdot \prod_{j \in \bN} \dfrac{1 + y_j \cdot \bmu_j}{1 + x_j \cdot \bmu_j}  = \bp(x) \left(1 - \frac{c_0 \cdot \eps}{\sqrt{n}} \right)^{2r} \left( 1 + \frac{c_0\cdot \eps}{\sqrt{n}}\right)^{-2r} \leq \bp(x) \left(1 - \frac{c_0\cdot \eps}{\sqrt{n}}\right)^{2r}.
\end{align*}
where the second expression is obtained by counting the number of coordinates $j \in \bN$ where $y_j$ is $1$ or $-1$, and comparing that to the number of coordinates $j \in \bN$ where $x_j$ is $1$ or $-1$, as well as the fact that every $j \in \bN$ has $\bmu_j= -c_0 \eps /\sqrt{n}$. Therefore, since $2r = 2(m/2 - |x_{\bN}|)$ for each $x$ above, we have
\begin{align*}
\sum_{(x,y) \in \calM} \left( \bp(x) - \bp(y) \right) &\geq \sum_{\substack{x \in \{-1,1\}^{n} \\ |x_{\bN}| < m/2 }} \bp(x) \left( 1 - \left( 1 - \frac{c_0 \cdot \eps}{\sqrt{n}}\right)^{2(m/2-|x_{\bN}|)}\right) \\
            &\geq \sum_{\substack{x \in \{-1,1\}^n \\ |x_{\bN}| < m/2}} \bp(x) \cdot \frac{2 c_0 \eps \cdot (m/2 - |x_{\bN}|)}{\sqrt{n}} = \frac{2c_0\eps}{\sqrt{n}}\cdot  \Ex_{\bx \sim \bp}\left[ \left(\frac{m}{2} - |\bx_{\bN}| \right)^+\right].
\end{align*}
Finally, note that because each $\bx_i$ when $\bx \sim \bp$ is independent, and for every $i \in \bN$, the probability that $\bx_i = 1$ is smaller than $1/2$ (recall $\bN$ are exactly the negatively-biased coordinates), standard anti-concentration inequalities imply that with constant probability over $\bx \sim \bp$, we have $|\bx_{\bN}| \leq m/2 - \sqrt{m} = \Omega(\sqrt{n})$. This gives the $\Omega(\eps)$ lower bound on the distance to monotonicity and finishes the proof of the lemma.
\end{proof}

\subsection{Indistinguishability of $\Dyes$ and $\Dno$}\label{sec:indist}

Recall that $\ell = \log n/\log\log n$. For convenience, we set $\alpha = \eps \ell^3/\sqrt{n}$,
the largest possible value of the mean vectors.

We show (\ref{eq:indistinguish}) with the following approach. First, we note that, since the algorithm receives $q$ independent samples from an $n$-dimensional product distribution, it suffices to consider draws to a ``vector of counts''. Namely, given a sequence of samples $x_1, \dots, x_q \in \{-1,1\}^n$, let $r(x_1, \dots, x_q)$ be the $n$-dimensional vector of integers, where
\[ r(x_1,\dots, x_q)_i = \sum_{k=1}^q \ind\{ (x_k)_i = 1 \}. \]
Then, we let $\calR_y$ (resp. $\calR_n$) denote the distribution given by (i) sampling $\bp \sim \Dyes$ (or, $\bp \sim \Dno$), (ii) letting $\bx_1,\dots, \bx_q \sim \bp$, and (iii) outputting $r(\bx_1,\dots, \bx_q)$. One may equivalently sample $\bx_1,\dots, \bx_q$ from $\Dyes$ or $\Dno$ by sampling $\boldr \sim \calR_y$ or $\calR_n$ and then generating the samples $\bx_1,\dots, \bx_q$ conditioned on the vector of counts $\boldr$. Thus, we will derive (\ref{eq:indistinguish}) by showing that 
$$\dtv\big(\calR_y, \calR_n\big) = o(1),
\quad\text{when $q\le \frac{n}{\eps^2\cdot
\text{polylog}(n)}$}.$$ 
Toward this end, we define $G \subset \Z_{\geq 0}^n$ (in Lemma \ref{lemma:hehe} below) and show that
\begin{itemize}
    \item Lemma \ref{lemma:hehe}:
 $\boldr \in G$ with probability at least $1 - o(1)$ over $\boldr \sim \calR_y$ and $\calR_n$; and
 \item Lemma \ref{lemma:haha}: 
For every count vector $x \in G$, we have
\begin{align*}
\dfrac{\Prx_{\boldr \sim \calR_y}\left[ \boldr = x \right]}{\Prx_{\boldr \sim \calR_n}\left[ \boldr = x \right]} = 1 \pm o(1).
\end{align*}
\end{itemize}
These two lemmas together imply the desired bound on the total variation distance.

\begin{lemma}\label{lemma:hehe}
    For any $q \in \N$, let $G \subset \Z_{\geq 0}^n$ denote the set of vectors $r$ 
      that sum to $q$ such that 
    \[ \frac{q}{2} - \frac{q\cdot \alpha}{2} 
 - \sqrt{q} \log n \leq r_i \leq \frac{q}{2} + \frac{q\cdot \alpha}{2} 
 + \sqrt{q} \log n \]
 for all $i\in [n]$.
 Then, for $\calR$ being both $\calR_y$ and $\calR_n$, we have
    \begin{align*}
        \Prx_{\boldr \sim \calR}\big[ \boldr \notin G \big] = o(1).
    \end{align*}
\end{lemma}

\begin{proof}
    Note that, each of the $n$ coordinates of $\boldr \sim \calR$ is independent and identically distributed; it is given by letting $\bmu \sim \Dyes$ (or $\Dno$), and then letting $\boldr_i \sim \Bin(q, \frac{1 + \bmu}{2})$. Since $\bmu$ is always below $\alpha$ and above $-\alpha$, so $\boldr_i$ and $-\boldr_i$ are stochastically dominated by $\Bin(q, (1+\alpha)/2)$. The bound then follows from a union bound and standard concentration arguments, using the fact that $\alpha$ is small (say, smaller than $0.01$).
\end{proof}

\begin{lemma}\label{lemma:haha}
Assume that $q \leq n / (\eps^2 \cdot \polylog(n))$.
For any $c \in G$, we have
\begin{align*}
    \dfrac{\Prx_{\boldr \sim \calR_y}\left[ \boldr = c \right]}{\Prx_{\boldr \sim \calR_n}\left[ \boldr = c \right]} = 1 \pm o(1).
\end{align*}
\end{lemma}

\begin{proof}
 For $\calR$ being $\calR_y$ or $\calR_n$ and $\calD$ being $\calD_y$ or $\calD_n$
   accordingly, 
   we write down $\Pr_{\boldr\sim \calR} [\boldr=c]$ as
   \begin{align*}
        \Prx_{\boldr \sim \calR}\left[ \boldr = c\right] &= \prod_{i=1}^n \Ex_{\bmu \sim \calD}\left[ \binom{q}{c_i} \left(\frac{1 + \bmu}{2}\right)^{c_i} \left( \frac{1 - \bmu}{2} \right)^{q-c_i} \right] \\
                        &= \prod_{i=1}^n \binom{q}{c_i} \cdot 2^{-q} \cdot \Ex_{\bmu \sim \calD}\left[ \left(1 - \bmu^2 \right)^{q/2 - d_i} \left(1+ \sigma_i \bmu \right)^{2 d_i} \right],
    \end{align*}
where $0\le d_i\le q/2$ and $\sigma_i \in \{-1,1\}$ uniquely satisfy $c_i - \sigma_i d_i = q/2$.

Note $c \in G$ implies $d_i \leq q \alpha/2 + \sqrt{q} \log n \leq \sqrt{n} / (\eps \cdot \polylog(n))$. 
        Furthermore, for any $\mu \in [-\alpha, \alpha]$, 
    \begin{align}
    \left(1 - \mu^2\right)^{q/2 - d_i}\left(1+ \sigma_i \mu\right)^{2d_i} = \sum_{\ell_1 = 0}^{q/2-d_i} \sum_{\ell_2 = 0}^{2d_i} \binom{q/2-d_i}{\ell_1} \binom{2d_i}{\ell_2} (-1)^{\ell_1} \sigma_i^{\ell_2} \mu^{2\ell_1 + \ell_2}. \label{eq:expression} \end{align}
    We consider the degree-$\ell$ Taylor expansion $T_{\ell}(\mu,d_i,\sigma_i)$ (with respect to $\mu$) at $0$ of (\ref{eq:expression}), and we have that the error in the expression becomes at most
    \begin{align*} 
    &\sum_{t > \ell} \sum_{\ell_1=0}^{q/2-d_i} \sum_{\ell_2=0}^{2d_i} \ind\{ 2\ell_1 + \ell_2 = t \} \cdot \binom{q/2-d_i}{\ell_1} \binom{2d_i}{\ell_2} \cdot \alpha^{t} \\
        &\qquad\qquad \leq \sum_{t > \ell} t\cdot \left( q^{1/2} + 2d_i\right)^{t} \cdot \alpha^{t} \leq \sum_{t > \ell} \left( 2\alpha \sqrt{q} + 4\alpha d_i  \right)^{t} 
        \leq \frac{1}{n^{10}},
    \end{align*}
    as long as $\alpha \sqrt{q} + \alpha d_i < 1 / \polylog(n)$ for a large enough polynomial, given that $\ell=\log n / \log \log n $, which it is by setting of $\alpha, q$ and $d_i$. Notice, furthermore, that (\ref{eq:expression}) is at least $1 - O(q\mu^2) - O(d_i \mu) \geq 1/2$. Hence, the ratio of the two probabilities, when $\calR$ is $\calR_y$ and $\calR_n$ is
    \begin{align*}
        \prod_{i=1}^n \dfrac{\Ex_{\bmu \sim \calD_y}\left[ T_{\ell}(\bmu, d_i,\sigma_i) \pm 1/n^{10}\right]}{\Ex_{\bmu \sim \calD_n}\left[ T_{\ell}(\bmu, d_i,\sigma_i) \pm 1/n^{10}\right]}=\prod_{i=1}^n \left(1\pm O(1/n^{10})\right) = 1 \pm O(1 / n^{9}),
    \end{align*}
    where we used the fact that the expected Taylor expansion to degree $\ell$ is equal for $\calD_y$ and $\calD_n$, since it is a function of the first $\ell$ moments of $\calD_y$ and $\calD_n$. 
\end{proof}


\section{Testing Uniformity of Monotone Distributions}

In this section, we prove the following theorem, which gives a query lower bound on testing uniformity of a distribution which is promised to be monotone using subcube conditioning queries. 
\begin{theorem}[Uniformity Testing of Monotone Distributions -- Lower Bound] For any $\eps > 0$,~any $\eps$-test for uniformity of distributions that are promised to be monotone must make $\tilde{\Omega}(\sqrt{n}/\eps^2)$ queries.
\end{theorem}
Similar to Section~\ref{sec:mon-lb}, we describe a distribution $\Dno$ supported on monotone product distributions over $\{-1,1\}^n$. Importantly, a distribution $\bp \sim \Dno$ will be $\Omega(\eps)$-far from the uniform distribution with probability $1-o_n(1)$ (Lemma \ref{lem:distancelemma}). Then, we show that for  $q$ which is at most $c \sqrt{n} / (\eps^2 \log^4 n)$, for a small enough constant $c > 0$, any function $\Alg \colon \{-1,1\}^{nq} \to \{\text{``accept''}, \text{``reject''} \}$ cannot output ``accept'' with probability at least $0.99$ when samples are drawn from the uniform distribution, and output ``reject'' with probability at least $0.99$ when samples are drawn from  $\Dno$.

\subsection{The distribution $\Dno$}\label{sec:dno-def}

A draw of $\bp \sim \Dno$ is generated as follows:
\begin{flushleft}\begin{itemize}
\item First, we let $\calD$ denote the distribution over vectors $\bmu$ where we 
independently set $\bmu_i$ to be $\eps/n^{1/4}$ with probability $1/\sqrt{n}$ and $0$ otherwise.
\item Then, we let $\bp$ be the monotone product distribution on $\{-1,1\}^n$ whose mean vector is $\bmu$.
\end{itemize}\end{flushleft}
The fact that $\bp \sim \Dno$ is far from the uniform distribution follows from the subsequent lemma.
\begin{lemma}\label{lem:distancelemma}
With probability at least $1 - o_n(1)$, $\bp \sim \Dno$ is $\Omega(\eps)$-far from the uniform distribution.
\end{lemma}
\begin{proof}
We consider a draw of $\bmu \sim \calD$, and note that with probability at least $1-o_n(1)$, $\bp \sim \Dno$ has a set $\bN \subset [n]$ of at least $\Omega(\sqrt{n})$ coordinates $i$ with $\bmu_i = \eps/n^{1/4}$. Fix such a draw and let $\bp$ denote the corresponding distribution. The total variation distance from $\bp$, generated with mean vector $\mu$, to the uniform distribution can be lower bounded by considering strings $x \in \{-1,1\}^n$ which have fewer $1$'s than $-1$'s in coordinates of $\bN$. 
For each such string $x$, letting $t(x)$ denote the number of coordinates in $\bN$ with $x_i = 1$, the probability of $x$ in $\bp$ is
$$
\frac{1}{2^n}\cdot \left(1 + \frac{\eps}{n^{1/4}} \right)^{t(x)} \left( 1 - \frac{\eps}{n^{1/4}}\right)^{|\bN| - t(x)} 
=\frac{1}{2^n}\cdot \left(1 - \frac{\eps^2}{\sqrt{n}}\right)^{t(x)} \left(1 - \frac{\eps}{n^{1/4}}\right)^{|\bN| - 2t(x)}< \frac{1}{2^n}.
$$
As a result, we can bounded $\dtv(\bp,\calU_n)$ from below as follows:
\begin{align}
\dtv(\bp, \calU_n) 
     &\geq \frac{1}{2^n} \sum_{\substack{x \in \{-1,1\}^n \\ t(x) \leq |\bN|/2}} \left( 1 - \left(1 - \frac{\eps^2}{\sqrt{n}}\right)^{t(x)} \left(1 - \frac{\eps}{n^{1/4}}\right)^{|\bN| - 2t(x)}\right) \nonumber\\[0.5ex]
    &\ge
     \frac{1}{2^n} \sum_{\substack{x \in \{-1,1\}^n \\ t(x) \leq |\bN|/2}} \left( 1 -  \left(1 - \frac{\eps}{n^{1/4}}\right)^{|\bN| - 2t(x)}\right).\label{eq:hehe1}
\end{align}
On the other hand, using $|\bN| = \Omega(\sqrt{n})$, there is a constant probability over a uniform $\bx \sim \{-1,1\}^n$ that $t(\bx)$ is bounded away from $|\bN|/2$ by $\Omega(n^{1/4})$.
For every such $\bx$, we have 
\[ 
 \left(1 - \frac{\eps}{n^{1/4}}\right)^{|\bN| - 2t(\bx)} \leq \left(1 - \frac{\eps}{n^{1/4}}\right)^{\Omega(n^{1/4})}\le e^{-\Omega(\eps)}\le 1-\Omega(\eps). \]
Combining everything we have from (\ref{eq:hehe1}) that
$$
\dtv(\bp,\calU_n)\ge \frac{1}{2^n}\cdot \Omega(2^n)\cdot \Omega(\eps)=\Omega(\eps).
$$
This finishes the proof of the lemma.
\end{proof}

\def\br{\mathbf{r}}

\subsection{Indistinguishability of $\Dno$ from the uniform distribution}

Consider the task of distinguishing via $q$ independent samples, $\bx_1, \dots, \bx_q \in \R^n$, whether these samples were drawn from the standard $n$-dimensional Gaussian $\calN(0, I)$, or an $n$-dimensional Gaussian $\calN(\bmu, I)$, where $\bmu \sim \calD$. We consider the above problem because of the following simple claim, which shows how to generate a product distribution whose mean vector has $i$-th coordinate $\Omega(\bmu_i)$.

\begin{claim}
Let $\sign\colon \R^n \to \{-1,1\}^n$ denote the function which applies $\sign(\cdot)$  coordinate-wise. 
\begin{flushleft}\begin{itemize}
\item The uniform distribution over $\{-1,1\}^n$ can be generated by sampling $\bx \sim \calN(0, I)$ and outputting $\sign(\bx)$.
\item For a fixed $\mu \in [0, 1/2]^n$, consider the product distribution over $\{-1,1\}^n$ generated by sampling $\bx \sim \calN( \mu, I)$ and outputting $\sign(\bx)$. Then, the mean vector of such a distribution has the $i$-th coordinate set to $\Omega(\mu_i)$.
\end{itemize}\end{flushleft}
\end{claim}

\begin{proof}
The first condition is by symmetry of the Gaussian distribution, and the second condition is by standard Gaussian anti-concentration, whenever $\mu \in [0, 1/2]$.
\end{proof}

Hence, it suffices to prove the following lemma.

\begin{lemma}\label{lem:mainlemma1}
Consider an algorithm that takes $q$ samples $\bx_1, \ldots, \bx_q\in \R^n$ and satisfies the following guarantees:
\begin{flushleft}\begin{itemize}
\item \textbf{\emph{Standard Case}}: If $\bx_1,\dots, \bx_q \sim \calN(0, I)$ and the algorithm receives those samples, then the algorithm outputs ``standard'' with probability at least $0.99$.
\item \textbf{\emph{Non-Standard Case}}: We sample $\bmu \sim \calD$,\footnote{The distribution $\calD$ is defined in the first bullet point of Subsection~\ref{sec:dno-def}.} then $\bx_1,\dots, \bx_q \sim \calN(\bmu, I)$, and the algorithm receives those samples, the algorithm outputs ``not standard'' with probability at least $0.99$.
\end{itemize}\end{flushleft}
Then, the number of samples must satisfy $q = \tilde{\Omega}(\sqrt{n}/\eps^2)$.
\end{lemma}

\subsection{Proof of \Lem{mainlemma1}} \label{sec:proof-mono-lb}

We prove by contradiction. {We assume that $q\le  \sqrt{n}/(c\eps^2 \log^4n)$ for some sufficiently large constant $c$} and show below
that the algorithm cannot distinguish between the two cases.

We set up some notation for the proof. We use $i \in [q]$ as an index over the set of queries,
while $j \in [n]$ indexes the dimension/coordinate.
We write $f_y, f_n \colon (\R^n)^q \to \R_{\geq 0}$ to denote the probability density functions of a tuple of $q$ independent samples from $\calN(0, I)$ (for $f_y$) or $\calN(\bmu, I)$ with $\bmu \sim \calD$ (for $f_n$) given by
\begin{align*}
f_y(x_1, \dots, x_q) &= \frac{1}{(2\pi)^{n/2}} \prod_{j=1}^n  \exp\left(-\frac{1}{2} \sum_{i=1}^q x_{ij}^2 \right) \qquad \text{and}\quad \\[1ex]
f_{n}(x_1,\dots, x_q) &= \frac{1}{(2\pi)^{n/2}} \prod_{j=1}^n \exp\left(-\frac{1}{2} \sum_{i=1}^q x_{ij}^2 \right) \left(1 - \frac{1}{\sqrt{n}} + \frac{1}{\sqrt{n}} \cdot \exp\left( \frac{\eps}{n^{1/4}}\sum_{i=1}^q x_{ij} -  \frac{q\eps^2}{2\sqrt{n}} \right)\right).
\end{align*}
The definition of $f_y$ comes from the product of $q$ many $n$-dimensional Gaussian p.d.fs; the definition of $f_n$ comes from the fact that each coordinate $j$ behaves independently under the draw of $\bmu \sim \calD$: with probability $1/\sqrt{n}$, $\bmu_i = \eps / n^{1/4}$ and is otherwise 0.

The main lemma below shows that these pdfs are nearly the same with high probability over draws $\bx_1,\ldots,\bx_q\sim \cN(0,I)$.
(Technically, we only need to lower bound $f_n$ by $f_y$.)

\begin{lemma} \label{lem:gaussian-calc} Consider $q$ independent draws of $\bx_i \sim \cN(0,I)$.
With probability at least $1- o_n(1)$, 
$$\frac{f_n(\bx_1, \ldots, \bx_q)}{f_y(\bx_1, \ldots, \bx_q)} \geq 1-o_n(1) .$$
\end{lemma}

\begin{proof} 
We set $\bX_j := \sum_{i=1}^q \bx_{ij}$ for each $j\in [n]$. 
Then the ratio can be written as
\begin{align}
\frac{f_n(\bx_1, \ldots, \bx_q)}{f_y(\bx_1, \ldots, \bx_q)}
&= \prod_{j=1}^n\left( 1 + \frac{1}{\sqrt{n}}\left( \exp\left(\frac{\eps\bX_j}{n^{1/4}}  - \frac{q\eps^2}{2\sqrt{n}}\right)-1\right)\right)
\end{align}
and thus, 
\begin{align}
 \ln\left( \frac{f_n(\bx_1, \ldots, \bx_q)}{f_y(\bx_1,  \ldots, \bx_q)}\right) 
&= 
  \sum_{j=1}^n \ln\left[ 1 + \frac{\bW_j}{\sqrt{n}}\right],
  \quad \text{with\ }\bW_j \eqdef \exp\left(\frac{\eps \bX_j}{n^{1/4}} - \frac{q\eps^2}{2\sqrt{n}}\right) - 1.
  \label{eq:wj}
\end{align}
%
At this point, we use the distributional information of $\bx_1, \ldots, \bx_n\sim \cN(0,I)$:

\begin{claim} \label{clm:wj} With probability at least $1-1/n$, we have 
$|\bW_j| \leq 1/\log n$ for all $j \in [n]$.
\end{claim}

\begin{proof} Note that $\bX_j = \sum_{i=1}^q \bx_{ij}$ where each $\bx_{ij} \sim \cN(0,1)$. Hence, $\bX_j \sim \cN(0,q)$.
With probability at least $1-1/n^2$, we have $|\bX_j| \le 4\sqrt{q}\log n$. By a union bound over all coordinates,~with probability at least $1-1/n$, we have $|\bX_j| \le  4\sqrt{q}\log n$ for all $j\in [n]$.

When this is the case, using $ {q\le \sqrt{n}/(c\eps^2 \log^4n)}$
  we have (when $c$ is sufficiently large)
$$\frac{\eps |\bX_j|}{n^{1/4}} \le \frac{4\eps \sqrt{q} \log n }{n^{1/4}} \le \frac{1}{4\log n}\quad\text{and}\quad
\frac{q\eps^2}{2\sqrt{n}} \le \frac{1}{4\log n}.$$ 
Using $1/(1-z) \geq \exp(z) \geq 1+z$ for $|z|\le 1$, we have
$$ \exp\left(\frac{\eps \bX_j}{n^{1/4}} - \frac{q\eps^2}{2\sqrt{n}}\right) \geq 1-\frac{1}{2\log n}$$
and
$$ \exp\left(\frac{\eps \bX_j}{n^{1/4}} - \frac{q\eps^2}{2\sqrt{n}}\right) \leq \frac{ 1}{1-1/(2\log n)} \leq 1 + \frac{1}{\log n}.$$
So $|\bW_j|  \leq 1/\log n$ for all $j$ and the claim follows.
\end{proof}

We go back to \Eqn{wj}, and apply the inequality $\ln(1+z) \geq z-z^2$ for $|z| \leq 1/2$.
With probability at least $1-1/n$ over $\bx_1, \ldots, \bx_n\sim \cN(0,I)$, we have $|\bW_j|\le 1/\log n$ for all $j$ and thus,
\begin{align}
\ln\left(\frac{f_n(\bx_1,  \ldots, \bx_q)}{f_y(\bx_1,  \ldots, \bx_q)} \right)\nonumber
&\geq \frac{1}{\sqrt{n}}\sum_{j=1}^n \bW_j - \frac{1}{n} \sum_{j=1}^n \bW^2_j \\
&\geq \frac{1}{\sqrt{n}}\sum_{j=1}^n \bW_j - \frac{1}{\log^2 n} \ \ \ ~~~~~~~ \textrm{(by \Clm{wj}, $\bW^2_j \leq 1/\log^2 n$)}\nonumber \\
&= \frac{1}{\sqrt{n}} \left[ \exp\left(-\frac{q \eps^2}{2\sqrt{n}}\right) \sum_{j=1}^n \exp\left(\frac{\eps \bX_j}{n^{1/4}}\right) - n \right]  - \frac{1}{\log^2 n}. \label{eq:gaussian}
\end{align}
The heart of the matter is the
next claim on the distribution of sum of exponentials of Gaussians.

\begin{claim} \label{clm:exp-gauss} Let each $\bX_j \sim \cN(0,q)$ be independent.
With probability at least $ 1-1/\sqrt{\log n}$,  we have$$\sum_{j=1}^n \exp\left(\frac{\eps \bX_j}{n^{1/4}}\right) \geq n \cdot\exp\left(\frac{q \eps^2}{2\sqrt{n}}\right) - \frac{\sqrt{n}}{\log^{0.25} n}.$$
\end{claim}

\begin{proof} Denote $\bY_j = \eps \bX_j/n^{1/4}$ and consider the random variable $\bZ_j = \exp(\bY_j)$. 
Observe~that~$\bY_j$ $\sim \cN(0, q\eps^2/\sqrt{n})$. Using the formula for the moment generating function
    of the Gaussian~\cite{gaussian-wiki}, we have $\EX[\exp(t\bY_j)] = \exp(q\eps^2 t^2/(2\sqrt{n}))$.
Hence, $$\EX[\bZ_j] = \EX[\exp(\bY_j)] = \exp\left(\frac{q\eps^2}{2\sqrt{n}}\right)\quad\text{and}\quad \EX[\bZ^2_j] = \EX[\exp(2\bY_j)] = \exp\left(\frac{2q\eps^2}{\sqrt{n}}\right).$$
So $\textrm{var}[\bZ_j] = \exp(2q\eps^2/\sqrt{n}) - \exp(q\eps^2/\sqrt{n})\leq 1/\log n$ using  $q\eps^2/\sqrt{n} = o (1/\log n)$.
Overall, we have 
$$\EX\left[\sum_{j=1}^n \bZ_j\right] = n\cdot \exp\left(\frac{q\eps^2}{ 2\sqrt{n}}\right)\quad\text{and}\quad\var\left[\sum_{j=1}^n \bZ_j\right] \leq \frac{n}{\log n}$$ since
all $\bZ_j$'s are independent. By Chebyshev's inequality, we have 
$$\Pr\left[\left|\sum_{j=1}^n \bZ_j - n\cdot \exp\left(\frac{q\eps^2}{2\sqrt{n}}\right)\right| > \frac{\sqrt{n}}{\log^{0.25} n}\right] \leq \frac{1}{\sqrt{\log n}}.$$
This finishes the proof of the claim.
\end{proof}

In particular, with probability at least $1-1/n - 1/\sqrt{\log n} = 1-o_n(1)$, we get that
\begin{align*}
\exp\left(-\frac{q \eps^2}{2\sqrt{n}}\right) \sum_{j=1}^n \exp\left(\frac{\eps \bX_j}{n^{1/4}}\right) & \geq \exp\left(-\frac{q \eps^2}{2\sqrt{n}}\right)\cdot \left( n \cdot \exp\left(\frac{q \eps^2}{2\sqrt{n}}\right) - \frac{\sqrt{n}}{\log^{0.25} n}\right)  >  n - \frac{\sqrt{n}}{\log^{0.25} n} 
    \end{align*}
where in the last inequality we used $\exp(-q \eps^2/2\sqrt{n}) < 1$. Substituting in \Eqn{gaussian}, we get
\begin{align}
\ln\left(\frac{f_n(\bx_1, \ldots, \bx_q)}{f_y(\bx_1,  \ldots, \bx_q)} \right)
& > -\frac{ 1}{    \log^{0.25} n}  - \frac{1}{\log^2 n} \label{eq:18}
\end{align}
Hence, with probability at least $1-o_n(1)$,  we have
$$\frac{f_n(\bx_1,  \ldots, \bx_q)}{f_y(\bx_1, \ldots, \bx_q)} \geq \exp\left(-\frac{ 1}{ \log^{0.25} n} - \frac{1}{\log^2n}\right) = 1-o_n(1).$$
This finishes the proof of the lemma.
\end{proof}

We can now complete the proof of \Lem{mainlemma1}. 
Consider the set $Y \subset (\R^n)^{q}$ of tuples that lead  the algorithm to output ``standard.''
Then we must have $$\Prx_{ \bx_i \sim \cN(0,I)}\big[(\bx_1, \dots, \bx_q) \in Y\big] \geq 0.99.$$ Let $Y' \subseteq Y$
be the set of tuples that also satisfy the condition of \Lem{gaussian-calc}.
By a union bound $$\Prx_{ \bx_i \sim \cN(0,I)}\big[(\bx_1, \dots, \bx_q) \in Y'\big] \geq 0.99 - o_n(1) \geq 0.98.$$
Thus, $\int_{Y'} f_y(\bx_1, \dots, \bx_q) d \bx_1 d \bx_2 \ldots d \bx_q \geq 0.98$. 
By the condition of \Lem{gaussian-calc}, we have  $$\int_{Y'}  f_n(\bx_1, \dots, \bx_q)  d \bx_1 \ldots d \bx_q \geq (1-o_n(1))\cdot0.98 \geq 0.97.$$
This is exactly the probability that we see a tuple in $Y'$,
when we generate the samples $\bx_1, \cdots, \bx_q$ from the non-standard case. Thus,
with probability at least $0.97$, the algorithm outputs ``standard'' when the samples are generated
from the non-standard case. This completes the contradiction.

\bibliographystyle{alpha}
\bibliography{waingarten,monotonicity-full}

\end{document}